\documentclass[12pt]{article}
\usepackage[height=1in, width=8.5in, margin=1in]{geometry}
\usepackage{setspace}
\usepackage{titlesec}
\titleformat*{\section}{\large\bfseries}
\titleformat*{\subsection}{\bfseries}
\titleformat*{\subsubsection}{\it}

\usepackage{amsmath,amssymb}
\usepackage{amsthm}
\usepackage{comment}
\usepackage{dcolumn}
\usepackage{graphicx}
\usepackage{import}
\usepackage{longtable}
\usepackage{mathrsfs}
\usepackage[version=3]{mhchem} 
\usepackage[numbers,sort&compress]{natbib}
\usepackage{pdfpages}
\usepackage{topcapt}
\usepackage{hyperref}
\usepackage{epstopdf}
\usepackage{url}
\usepackage{caption}
\usepackage{mathptmx}      

\newcommand{\fRelax}[4]{f^{#1}_{#2}(#3,#4)}
\newcommand{\E}[1]{\mathbb{E}[#1]}
\renewcommand{\P}[1]{\mathbb{P}(#1)}

\newcommand{\R}{\mathbb{R}}

\newcommand{\sumi}{\sum\limits_{i=1}^n}

\newcommand{\bom}{\boldsymbol{\omega}}

\newcommand{\bgam}{\boldsymbol{\gamma}}
\newcommand{\bs}[1]{{\mathbf{#1}}}
\newcommand{\x}{\mathbf{x}}
\newcommand{\ovX}{\overline{X}}
\newcommand{\ovOmega}{\overline{\Omega}}
\newcommand{\ovW}{\overline{W}}
\newcommand{\ovWi}{\overline{W}_i}
\newcommand{\commentout}[1]{}

\newcounter{Ex_Count}
\newtheorem{theorem}{Theorem}[section]
\newtheorem{lemma}[theorem]{Lemma}

\newtheorem{corollary}[theorem]{Corollary}
\newtheorem{assumption}[theorem]{Assumption}

\theoremstyle{definition}
\newtheorem{definition}[theorem]{Definition}

\newtheorem{exm}[Ex_Count]{Example}

\theoremstyle{remark}
\newtheorem{remark}[theorem]{Remark}

\title{Convex Relaxations for Global Optimization Under Uncertainty Described by Continuous Random Variables}

\author{Yuanxun Shao and Joseph K. Scott\footnotemark}

\begin{document}

    \maketitle
    \begin{abstract}
    This article considers nonconvex global optimization problems subject to uncertainties described by continuous random variables. Such problems arise in chemical process design, renewable energy systems, stochastic model predictive control, etc. Here, we restrict our attention to problems with expected-value objectives and no recourse decisions. In principle, such problems can be solved globally using spatial branch-and-bound (B\&B). However, B\&B requires the ability to bound the optimal objective value on subintervals of the search space, and existing techniques are not generally applicable because expected-value objectives often cannot be written in closed-form. To address this, this article presents a new method for computing convex and concave relaxations of nonconvex expected-value functions, which can be used to obtain rigorous bounds for use in B\&B. Furthermore, these relaxations obey a second-order pointwise convergence property, which is sufficient for finite termination of B\&B under standard assumptions. Empirical results are shown for three simple examples.
\end{abstract}

\setcounter{footnote}{1}
\renewcommand*{\thefootnote}{\fnsymbol{footnote}}
\footnotetext{Department of Chemical and Biomolecular Engineering, Clemson University, Clemson, SC (\url{jks9@clemson.edu}) } 

\section{Introduction}
\label{Introduction}
This article presents a new method for automatically constructing convex underestimators and concave overestimators (i.e., convex and concave relaxations) of functions of the form $F(\x)\equiv \E{f(\x,\bom)}$, where $\mathbb{E}$ denotes the expected value over continuous random variables $\bom\in\overline{\Omega}\subset\R^{n_{\omega}}$ and $f$ may be nonconvex in both arguments. The ability to construct convex and concave relaxations of nonconvex functions is central to algorithms for solving nonlinear optimization problems to guaranteed global optimality \cite{Horst:GOtext1,Sahinidis:ConvRel,McCormick:1976,Adjiman1998}. Here, we are specifically motivated by the following global optimization problem under uncertainty:
\begin{alignat}{1}
\label{Eq:IntroOpt}
\min_{\bs{x}\in C} \ \E{f(\x,\bom)}
\end{alignat}
Such problems arise broadly in chemical process design \cite{Behrooz2017}, structural design \cite{LIEM2017}, renewable energy systems \cite{Hakizimana:2016,Papadopoulos:2013}, portfolio optimization \cite{Pang2015}, stochastic model predictive control \cite{FARINA2015,Mesbah2016}, discrete event systems \cite{Ermoliev:1997}, etc. Moreover, although we restrict our attention to single-stage problems in this article (i.e., problems with no recourse decisions),  more flexible two-stage and multistage formulations can also be reduced to \eqref{Eq:IntroOpt} through the use of parameterized decision rules, which is an increasingly popular method for obtaining tractable approximate solutions \cite{Georghiou2015}.

In the applications above, many uncertainties are best modeled by continuous random variables, including process yields, material properties, renewable power generation, product demands, returns on investments, etc.~\cite{LIEM2017,Hakizimana:2016,Papadopoulos:2013,Pang2015}. However, when $f$ is nonlinear, this very often precludes writing the function $F(\x)\equiv\E{f(\x,\bom)}$ analytically in closed form. Moreover, expressing $F(\x)$ via quadrature rules quickly becomes intractable as the dimension of $\bom$ increases. In this situation, $F(\x)$ must be evaluated by sampling, which fundamentally limits the applicability of guaranteed global optimization algorithms such as spatial Branch-and-Bound (B\&B). Specifically, to perform an exhaustive global search, B\&B requires the ability to compute guaranteed lower and upper bounds on the optimal value of $F$ on any given subinterval $X$ of its domain. Assuming for simplicity that the set $C$ in \eqref{Eq:IntroOpt} is convex, these bounds are typically computed by minimizing a convex relaxation of $F$ over $X\cap C$ (for the lower bound) and evaluating $F$ at a feasible point in $X\cap C$ (for the upper bound). However, conventional methods for constructing convex relaxations are only applicable to functions that are known explicitly in closed form \cite{Horst:GOtext1,Sahinidis:ConvRel,McCormick:1976,Adjiman1998}. Moreover, using only sample-based approximations of $F$, it is not even possible to bound its value at a single feasible point with finitely many computations.

At present, the most common approach for solving \eqref{Eq:IntroOpt} globally is sample-average approximation (SAA), which approximates $F$ using fixed samples of the random variables chosen \emph{prior} to optimization. This results in a deterministic optimization problem that can be solved globally using conventional methods \cite{Horst:GOtext1,Sahinidis:ConvRel,McCormick:1976,Adjiman1998}. However, SAA has several critical limitations that often lead to inaccurate solutions or excessive computational cost for nonconvex problems. Most notably, it only guarantees convergence to a global solution (with probability one) as the sample size tends to infinity \cite{Shapiro2008}. Moreover, the number of samples required to achieve a high-quality solution in practice is unknown and can be quite large \cite{Verweij2003} (also see `ill-conditioned' problems in \cite{Linderoth2006}). Perhaps more importantly, a sufficient sample size is typically not known \emph{a priori}. Theoretical bounds are available \cite{Shapiro2006}, but are not generally computable and are often excessively large \cite{Shapiro2008}. Instead, state-of-the-art SAA methods solve multiple deterministic approximations with independent samples to compute statistical bounds on the approximation error \cite{Linderoth2006}. However, these are only asymptotically valid. Moreover, this typically involves solving 10--40 independent instances to global optimality, and the entire procedure must be repeated from scratch if a larger sample size is deemed necessary \cite{Linderoth2006}. This is clearly problematic for nonconvex problems, where solving just one instance to global optimality is already computationally demanding.

An alternative approach is the stochastic branch-and-bound algorithm \cite{Norkin1998}, which applies spatial B\&B to \eqref{Eq:IntroOpt} using probabilistic upper and lower bounds on each node $X$ based on a finite number of samples. Relative to SAA, a strength of this approach is that the sample size can be dynamically adapted as the search proceeds. However, since the computed bounds are only statistically valid, there is a nonzero probability of fathoming optimal solutions. Thus, stochastic B\&B only ensures convergence to a global solution when no fathoming is done, and even then only in the limit of infinite branching.

In the case where $f$ is convex, a number of deterministic (i.e., sample-free) methods are available for computing rigorous underestimators and overestimators of $F(\x)=\E{f(\x,\bom)}$ that can be used for globally solving \eqref{Eq:IntroOpt}. The simplest underestimator is given by Jensen's inequality \cite{Perlman:1974}, which states that $\E{f(\x,\bom)} \geq f(\x,\E{\bom})$ for convex $f$. This leads to a deterministic lower bounding problem for \eqref{Eq:IntroOpt} that can be solved using standard methods \cite{Birge:1986}. Moreover, $F(\x)$ can be bounded above at any feasible point $\x$ using, e.g., the Edmunson-Madansky upper bound \cite{Madansky:1959}, which uses values of $f(\x,\bom)$ at the extreme points of the uncertainty set $\overline{\Omega}$. Notably, these upper and lower bounds can be made arbitrarily tight using successively refined partitions of $\overline{\Omega}$, which allows \eqref{Eq:IntroOpt} to be solved to guaranteed $\epsilon$-global optimality \cite{Birge:1986,Frauendorfer:2011}. Moreover, refinements of the Jensen and Edmunson-Madansky bounds using higher order moments of $\bom$ have been studied in \cite{Dokov:2002,Edirisinghe1996}. However, all of these methods require $f$ to be convex or, more generally, to satisfy a convex-concave saddle property. Moreover, achieving convergence by partitioning $\overline{\Omega}$ requires the ability to efficiently compute probabilities and conditional expectations on subintervals of $\overline{\Omega}$, which severely limits the distributions of $\bom$ that can be handled efficiently (e.g., to multivariate uniform or Gaussian distributions with independent elements).

To address these limitations, this article presents a new method for computing deterministic convex and concave relaxations of $F(\x)=\E{f(\x,\bom)}$ on subintervals $X$ of its domain. This method applies to arbitrary nonconvex functions $f$ and a very general class of multivariate probability density functions for $\bom$. In brief, nonconvexity is addressed through a novel combination of Jensen's inequality and existing relaxation techniques for deterministic functions, such as McCormick's technique \cite{Scott:GenMCRel}. Similarly, general multivariate densities are addressed by relaxing nonconvex transformations that relate $\bom$ to random variables with simpler distributions. We show that the resulting relaxations of $F$ can be improved by partitioning $\overline{\Omega}$. More importantly, we establish second-order pointwise convergence of the relaxations to $F$ as the diameter of $X$ tends to $0$, provided that the partition of $\overline{\Omega}$ is appropriately refined.

The proposed relaxations can be used to compute both upper and lower bounds for \eqref{Eq:IntroOpt} restricted to any given $X$. In principle, this enables the global solution of \eqref{Eq:IntroOpt} by spatial B\&B. We leave the details of this algorithm for future work. However, we note here some potentially significant advantages of this approach relative to existing methods. First, unlike the stochastic B\&B algorithm, our relaxations provide deterministic upper and lower bounds for \eqref{Eq:IntroOpt} on any $X$, so there is zero probability of incorrectly fathoming an optimal solution during B\&B. Second, the convergence of our relaxations as the diameter of $X$ tends to $0$ implies that spatial B\&B will terminate finitely with an $\epsilon$-global solution under standard assumptions \cite{Horst:GOtext1,Bompadre:2012}. This is in contrast to both stochastic B\&B and SAA, which only converge in the limit of infinite sampling. In fact, the second-order convergence rate established here is known to be critical for avoiding the so-called \emph{cluster effect} in B\&B, and hence avoiding exponential run-time in practice \cite{Kannan2017}. Third, although refining the partition of $\overline{\Omega}$ is required for convergence, valid relaxations on $X$ can be obtained using any partition of $\overline{\Omega}$, no matter how coarse. Thus, the partition of $\overline{\Omega}$ can be adaptively refined as the B\&B search proceeds. This is potentially a very significant advantage over SAA, which must determine an appropriate number of samples \emph{a priori}, and must solve the problem again from scratch if more samples are deemed necessary.

The remainder of this article is organized as follows. Section \ref{Preliminaries} establishes the necessary definitions and notation, followed by a formal problem statement in Section \ref{Problem Statement}. Next, Sections \ref{Relaxations for Expected-Value Functions} and \ref{Convergence} present the main theoretical results establishing the validity and convergence of the proposed relaxations, respectively. Section \ref{Non-Uniform RVs} develops an extension of the basic relaxation technique that enables efficient computations with a wide variety of multivariate probability density functions. Finally Section \ref{Conclusions} concludes the article.

\section{Preliminaries}
\label{Preliminaries}
Convex and concave relaxations are defined as follows.
\begin{definition}
\label{Def:Relaxation of functions}
Let $S\subset \R^n$ be convex and $h:S\rightarrow \R$. Functions $h^{cv},h^{cc}:S\rightarrow \R$ are \emph{convex and concave relaxations of $h$ on $S$}, respectively, if $h^{cv}$ is convex on $S$, $h^{cc}$ is concave on $S$, and 
\begin{alignat*}{1}
h^{cv}(\bs{s})\leq h(\bs{s})\leq h^{cc}(\bs{s}), \quad \forall \bs{s}\in S.
\end{alignat*}
\end{definition}

The following extension of Definition \ref{Def:Relaxation of functions} is useful for convergence analysis. First, for any $\bs{s}^L,\bs{s}^U\in\mathbb{R}^{n}$ with $\bs{s}^L\leq\bs{s}^U$, let $S=[\bs{s}^L,\bs{s}^U]$ denote the compact $n$-dimensional interval $\{\bs{s}\in\mathbb{R}^{n}:\bs{s}^L\leq\bs{s}\leq\bs{s}^U\}$. Moreover, for $\overline{S}\subset \mathbb{R}^{n}$, let $\mathbb{I}\overline{S}$ denote the set of all compact interval subsets $S$ of $\overline{S}$. In particular, let $\mathbb{IR}^n$ denote the set of all compact interval subsets of $\mathbb{R}^n$. Finally, denote the \emph{width} of $S$ by $w(S)=\max\limits_i\vert s_i^U-s_i^L\vert$.

\begin{definition}
\label{Def:scheme of relaxations}
Let $\overline{S}\subset \R^n$ and $h:\overline{S}\rightarrow \R$. A collection of functions $h_{S}^{cv},h_{S}^{cc}: \overline{S}\rightarrow \R$ indexed by $S\in\mathbb{I}\overline{S}$ is called a \emph{scheme of relaxations for $h$ in $\overline{S}$} if, for every $S\in\mathbb{I}\overline{S}$, $h_{S}^{cv}$ and $h_{S}^{cc}$ are convex and concave relaxations of $h$ on $S$.
\end{definition}

\begin{remark}
The notion of a scheme of relaxations was first introduced in \cite{Bompadre:2012} with the alternative name \emph{scheme of estimators}. Here, we use the term \emph{relaxations} in place of \emph{estimators} to avoid possible confusion with the common meaning of \emph{estimation} in the stochastic setting.
\end{remark}

The following notion of convergence for schemes of relaxations originates in \cite{Bompadre:2012}.

\begin{definition}
\label{pointwise convergence}
Let $\overline{S}\subset \R^n$ and $h:\overline{S}\rightarrow \R$. A scheme of relaxations $(h^{cv}_{S},h^{cc}_{S})_{{S}\in\mathbb{I}\overline{S}}$ for $h$ in $\overline{S}$ has \emph{pointwise convergence of order $\gamma$} if $\exists\tau\in\mathbb{R}_+$ such that
\begin{alignat*}{1}
\sup\limits_{\bs{s}\in {S}}\vert h^{cc}_{S}(\bs{s})-h^{cv}_{S}(\bs{s})\vert\leq \tau w(S)^{\gamma}, \quad \forall S\in\mathbb{I}\overline{S}.
\end{alignat*}
\end{definition}

Note that the constants $\gamma$ and $\tau$ in Definition \ref{pointwise convergence} may depend on $\overline{S}$, but not on $S$. First-order convergence is necessary for finite termination of spatial-B\&B algorithms, while second-order convergence is known to be critical for efficient B\&B because it can eliminate the \emph{cluster effect}, which refers to the accumulation of a large number of B\&B nodes near a global solution \cite{Horst:GOtext1,Bompadre:2012,Kannan2017}.

Several methods are available for automatically computing schemes of relaxations for \emph{factorable functions}. A function $h$ is called \emph{factorable} if it is a finite recursive composition of basic operations including $\{+,-,\times,\div\}$ and standard univariate functions such as $x^{n}$, $e^x$, $\sin{x}$, etc. Roughly, factorable functions include every function that can be written explicitly in computer code. Schemes of relaxations for factorable functions can be computed by the $\alpha$BB method \cite{Adjiman1998,Skjal2012}, McCormick's relaxation technique \cite{McCormick:1976,Scott:GenMCRel}, the addition of variables and constraints \cite{Sahinidis:ConvRel}, and several advanced techniques \cite{Meyer2005,Gounaris2008,Bao:2015,Tsoukalas2014,Khan2017}. Moreover, both $\alpha$BB and McCormick relaxations are known to to exhibit second-order pointwise convergence \cite{Bompadre:2012}. However, for functions $h$ that are not known in closed form, and hence are not factorable, none of the aforementioned techniques apply. Some recent extensions of $\alpha$BB and McCormick relaxations do address certain types of implicitly defined functions, such as the parametric solutions of fixed-point equations \cite{Wechsung:2015,Stuber:GOIF}, ordinary differential equations \cite{Fl:BBB,Scott:2GODERel}, and differential-algebraic equations \cite{Scott:NonlinDAERels}. However, no techniques are currently available for relaxing the expected-value function of interest here.

\section{Problem Statement}
\label{Problem Statement}
Let $\bom\in\mathbb{R}^{n_{\omega}}$ be a vector of continuous random variables (RVs) distributed according to a probability density function (PDF) $p:\mathbb{R}^{n_{\omega}}\rightarrow\mathbb{R}$. We assume throughout that $p$ is zero outside of a compact interval $\overline{\Omega}\subset \R^{n_{\omega}}$. Let $\overline{X}\subset \mathbb{R}^{n_x}$ be compact, let $f:\overline{X}\times \overline{\Omega}\rightarrow\mathbb{R}$ be a potentially nonconvex function, and assume that the expected value $\E{f(\x,\bom)}\equiv\int_{\ovOmega}f(\x,\bom)p(\bom)d\bom$ exists for all $\x\in \ovX$. Moreover, define $F:\ovX\rightarrow\mathbb{R}$ by $F(\x)\equiv\E{f(\x,\bom)}$, $\forall \x\in \ovX$.

The objective of this article is to present a new scheme of relaxations for $F$ in $\ovX$ with second-order pointwise convergence. In particular, this scheme addresses the general case where $F$ cannot be expressed explicitly as a factorable function of $\x$, and standard relaxation techniques cannot be applied. In such cases, $F$ is most often approximated via sampling and, in general, $F(\x)$ cannot even be evaluated exactly with finitely many computations. Critically, our new scheme consists of relaxations that provide bounds on $F$ itself, rather than a finite approximation of $F$, but are nonetheless finitely computable. Therefore, these relaxations can be used within a spatial-B\&B framework to solve \eqref{Eq:IntroOpt} to $\epsilon$-global optimality without approximation errors. 

A central assumption in the remainder of the article is that the integrand $f$ is a factorable function, or that a scheme of relaxations is available by some other means.
\begin{assumption}
\label{Assump:f factorable}
A scheme of relaxations for $f$ in $\ovX\times\ovOmega$, denoted by $(f_{X\times\Omega}^{cv},f_{X\times\Omega}^{cc})_{X\times\Omega\in\mathbb{I}\ovX\times\mathbb{I}\ovOmega}$, is available. $f_{X\times\Omega}^{cv}$ and $f_{X\times\Omega}^{cc}$ denote convex and concave relaxations of $f$ jointly on $X\times\Omega$.
\end{assumption}

\begin{remark}
We will sometimes make use of relaxations of $f$ with respect to either $\x$ or $\bom$ independently, with the other treated as a constant. We denote these naturally by $f^{cv}_{X}(\x,\bom)$ and $f^{cv}_{\Omega}(\x,\bom)$. Within the scheme of Assumption \ref{Assump:f factorable}, these relaxations are equivalent to the more cumbersome notations $f^{cv}_{X}(\x,\bom)\equiv f^{cv}_{X\times[\bom,\bom]}(\x,\bom)$ and $f^{cv}_{\Omega}(\x,\bom)\equiv f^{cv}_{[\x,\x]\times\Omega}(\x,\bom)$.
\end{remark}

\section{Relaxing Expected-Value Functions}
\label{Relaxations for Expected-Value Functions}
This section presents a general approach for constructing finitely computable convex and concave relaxations of the expected value function $F(\x)=\E{f(\x,\bom)}$. Let $(f_{X\times\Omega}^{cv},f_{X\times\Omega}^{cc})$ be a scheme of relaxations for $f$ on $\ovX\times \ovOmega$ as per Assumption \ref{Assump:f factorable} and choose any $X\in\mathbb{I}\ovX$. To begin, note that a direct application of integral monotonicity gives (\cite{Karr:1993}, p.101)
\begin{alignat}{1}
\label{Eq:Integral monotonicity}
\E{\fRelax{cv}{X}{\x}{\bom}}\leq \E{f(\x,\bom)} \leq \E{\fRelax{cc}{X}{\x}{\bom}},\quad \forall \x\in X,
\end{alignat}
which suggests defining relaxations of $F$ by $F^{cv}(\x)\equiv\E{\fRelax{cv}{X}{\x}{\bom}}$ and $F^{cc}(\x)\equiv\E{\fRelax{cc}{X}{\x}{\bom}}$. However, although these functions are indeed convex and concave on $X$, respectively, they are not finitely computable because they need to be evaluated by sampling in general. To overcome this limitation, we apply Jensen's inequality, which is stated as follows.

\begin{lemma}
\label{Jensen's Inequality}
Let $\Omega\subset\overline{\Omega}$ be convex and let $g:\Omega\rightarrow \R$. If $g$ is convex and $\E{g(\bom)}$ exists, then $\E{g(\bom)}\geq g(\E{\bom})$. If $g$ is concave, then $\E{g(\bom)}\leq g(\E{\bom})$.
\end{lemma}
\begin{proof}
See Proposition 1.1 in \cite{Perlman:1974}.
\end{proof}

Although Jensen's inequality is widely used to relax stochastic programs, it has so far only been applied in the case where $f(\x,\cdot)$ is convex on $\ovOmega$ for all $\x\in X$, which we do not assume here. Instead, we propose to combine existing convex relaxation techniques such as McCormick relaxations with Jensen's inequality. To do this, it is necessary to relax $f$ \emph{jointly} on $X\times \ovOmega$. Then, integral monotonicity and Jensen's inequality imply that
\begin{alignat}{1}
\label{Eq:Integral monotonicity plus Jensens cv}
\E{f(\x,\bom)} \geq \E{\fRelax{cv}{X\times\ovOmega}{\x}{\bom}} \geq \fRelax{cv}{X\times\ovOmega}{\x}{\E{\bom}},\\
\label{Eq:Integral monotonicity plus Jensens cc}
\E{f(\x,\bom)} \leq \E{\fRelax{cc}{X\times\ovOmega}{\x}{\bom}}\leq \fRelax{cc}{X\times\ovOmega}{\x}{\E{\bom}}.
\end{alignat}
for all $\x\in X$. Note that the integrals $\E{\fRelax{cv}{X\times\ovOmega}{\x}{\bom}}$ and $\E{\fRelax{cc}{X\times\ovOmega}{\x}{\bom}}$ exist because the convexity and concavity of the integrands implies that they are continuous on the interior of $X\times\ovOmega$, and the boundary of $X\times\ovOmega$ has measure zero because $X\times\ovOmega$ is an interval. The inequalities \eqref{Eq:Integral monotonicity plus Jensens cv}--\eqref{Eq:Integral monotonicity plus Jensens cc} suggest defining relaxations for $F$ by $F^{cv}(\x)\equiv \fRelax{cv}{X\times\ovOmega}{\x}{\E{\bom}}$ and $F^{cc}(\x)\equiv\fRelax{cc}{X\times\ovOmega}{\x}{\E{\bom}}$. These relaxations are clearly convex and concave on $X$, respectively, and are finitely computable provided that $\E{\bom}$ is known. However, a remaining difficulty is that the under/over-estimation caused by the use of Jensen's inequality does not to converge to zero as $w(X)\rightarrow0$, which is required for finite termination of spatial-B\&B. We address this problem by considering relaxations constructed on interval partitions of $\ovOmega$.

\begin{definition}
A collection $\mathcal{P}=\{ \Omega_i \}_{i=1}^{n}$ of compact intervals $\Omega_i\in\mathbb{I}\ovOmega$ is called an \emph{interval partition of $\ovOmega$} if $\ovOmega=\cup_{i=1}^n\Omega_i$ and $\mathrm{int}(\Omega_i)\cap\mathrm{int}(\Omega_j)=\emptyset$ for all distinct $i$ and $j$.
\end{definition}

\begin{definition}
For any measurable $\Omega\subset\ovOmega$, let $\mathbb{P}(\Omega)$ denote the probability of the event $\bom\in\Omega$, and let $\mathbb{E}[\cdot|\Omega]$ denote the conditional expected-value conditioned on the event $\bom\in\Omega$.
\end{definition}

The following theorem extends the relaxations defined above to partitions of $\ovOmega$.
\begin{theorem}
\label{relaxations of RVs}
Let $\mathcal{P}=\{ \Omega_i \}_{i=1}^{n}$ be an interval partition of $\ovOmega$. For every $X\in\mathbb{I}\ovX$ and every $\x\in X$, define
\begin{alignat}{1}
\label{Eq:FcvXP definition}
F^{cv}_{X \times \mathcal{P}}(\x) &\equiv  \sumi \P{\Omega_i}\fRelax{cv}{X\times \Omega_i}{\x}{\mathbb{E}[\bom| \Omega_i]}, \\
\label{Eq:FccXP definition}
F^{cc}_{X \times \mathcal{P}}(\x) &\equiv \sumi \P{\Omega_i}\fRelax{cc}{X\times \Omega_i}{\x}{\mathbb{E}[\bom| \Omega_i]}.
\end{alignat}
$F^{cv}_{X \times \mathcal{P}}$ and $F^{cc}_{X \times \mathcal{P}}$ are convex and concave relaxations of $F$ on $X$, respectively.
\begin{proof}
By the law of total expectation (Proposition 5.1 in \cite{Ross:2002}), 
\begin{alignat}{1}
\label{law of total expectation for F(x)}
F(\x) & =\E{f(\x,\bom)}=\sumi \P{\Omega_i}\mathbb{E}[f(\x,\bom)|\Omega_i],
\end{alignat}
for all $\x\in \ovX$. Thus, for any $X\in\mathbb{I}\ovX$ and any $\x\in X$, integral monotonicity and Jensen's inequality give,
\begin{alignat}{1}
F(\x) & \geq \sumi \P{\Omega_i}\mathbb{E}[\fRelax{cv}{X\times \Omega_i}{\x}{\bom} | \Omega_i], \\
&\geq \sumi \P{\Omega_i}\fRelax{cv}{X\times \Omega_i}{\x}{\mathbb{E}[\bom | \Omega_i]}, \\
&=F^{cv}_{X\times\mathcal{P}}(\x).
\end{alignat}
Moreover, $F^{cv}_{X\times\mathcal{P}}$ is convex on $X$ because it is a sum of convex functions. Thus, $F^{cv}_{X\times\mathcal{P}}$ is a convex relaxation of $F$ on $X$, and the proof for $F^{cc}_{X\times\mathcal{P}}$ is analogous. 
\end{proof}
\end{theorem}

The relaxations \eqref{Eq:FcvXP definition}--\eqref{Eq:FccXP definition} are finitely computable provided that the probabilities $\mathbb{P}(\Omega_i)$ and conditional expectations $\mathbb{E}[\bom|\Omega_i]$ are computable for any subinterval $\Omega_i\subset\ovOmega$. This is trivial when $\bom$ is uniformly distributed, but requires difficult multidimensional integrations even for Gaussian random variables. We develop a general approach for avoiding such integrals for a broad class of RVs called \emph{factorable RVs} in \S\ref{Non-Uniform RVs}. Given these probabilities and expected values, the required relaxations of the integrand $f$ can be computed using any standard technique. In the following example we apply McCormick relaxations \cite{McCormick:1976,Scott:GenMCRel}. In this case, we call the relaxations \eqref{Eq:FcvXP definition}--\eqref{Eq:FccXP definition} \emph{Jensen-McCormick (JMC) relaxations}.

\begin{exm}
\label{Ex: Simple JMC}
Let $\ovX\equiv[-1,1]\times[-1,1]$, let $\bom=(\omega_1,\omega_2)$ be uniformly distributed in $\ovOmega\equiv[0,1]\times[0,2]$, and define $F:\ovX\rightarrow\mathbb{R}$ by $F(\x)\equiv \E{f(\x,\bom)}$ with
\begin{alignat}{1}
f(\x,\bom)\equiv\frac{x_1x_2\ln(3+x_1\omega_1\omega_2)-(x_1^2-1)(x_2^2-1)\omega_2^2}{2+\omega_1x_1}. \nonumber
\end{alignat}
The nonlinearity of $f$ makes it difficult if not impossible to evaluate $F$ analytically. Nonetheless, a rigorous convex relaxation for $F$ can be constructed using Theorem \ref{relaxations of RVs}. Figure \ref{Fig MCJ Relaxations} shows the relaxation $F^{cv}_{\ovX \times \mathcal{P}}(\x)$ computed using three different partitions $\mathcal{P}$ of $\ovOmega$ with $1$, $16$, and $64$ uniform subintervals each. The required relaxations $f_{\ovX\times\Omega_i}^{cv}$ were automatically constructed on each $\ovX\times\Omega_i$ by McCormick's relaxation technique \cite{McCormick:1976,Scott:GenMCRel}. Figure \ref{Fig MCJ Relaxations} also shows simulated $f(\x,\bom)$ values and a sample-average approximation of $F$ using 100 samples. Clearly, the Jensen-McCormick relaxations are convex and underestimate the expected value $F(\x)$. Interestingly, however, they do not underestimate $f(\x,\bom)$ for every $\bom\in\ovOmega$. Figure \ref{Fig MCJ Relaxations} also shows that $F_{\ovX\times\mathcal{P}}^{cv}(\x)$ gets significantly tighter as the partition of $\ovOmega$ is refined from 1 subinterval to 16, while the additional improvement from 16 to 64 is small. This shows that sharp results are obtained with few subintervals in this case. Note that $F_{\ovX\times\mathcal{P}}^{cv}(\x)$ will not converge to $F(\x)$ under further $\ovOmega$ partitioning unless $\ovX$ is also partitioned, as in spatial B\&B. \qed

\begin{figure}
\centering
\includegraphics[width=0.50\textwidth]{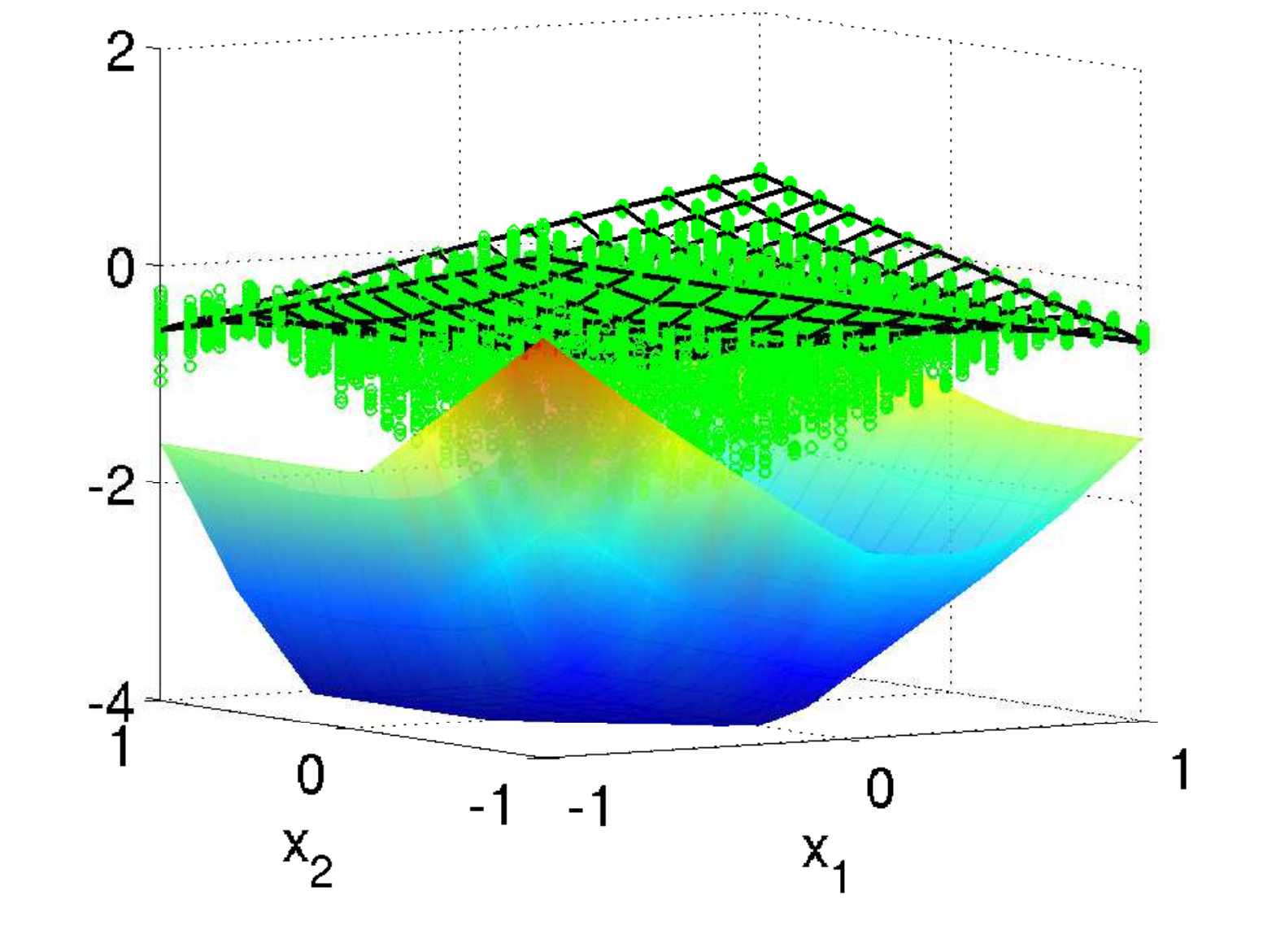}
\includegraphics[width=0.50\textwidth]{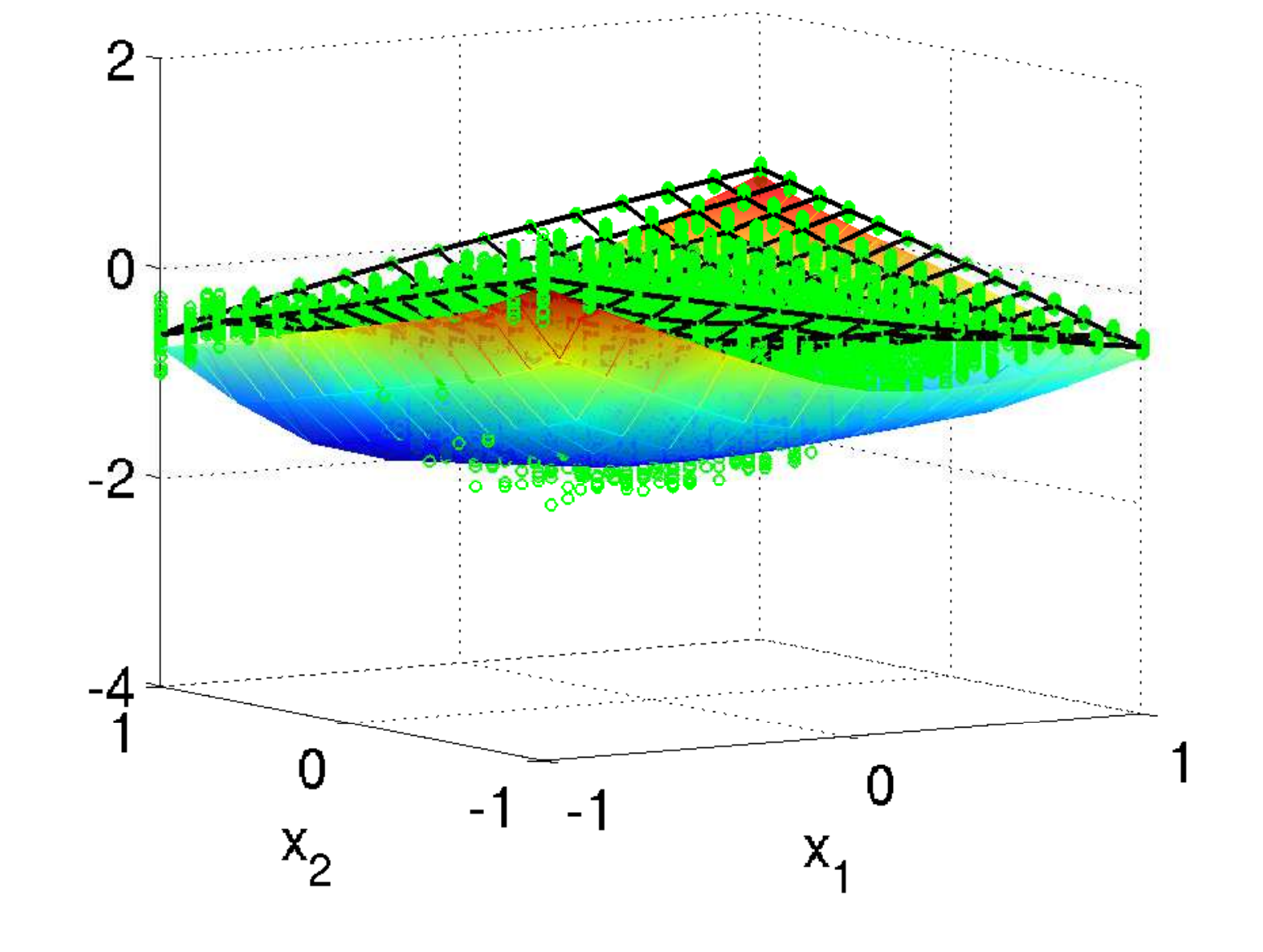}
\includegraphics[width=0.50\textwidth]{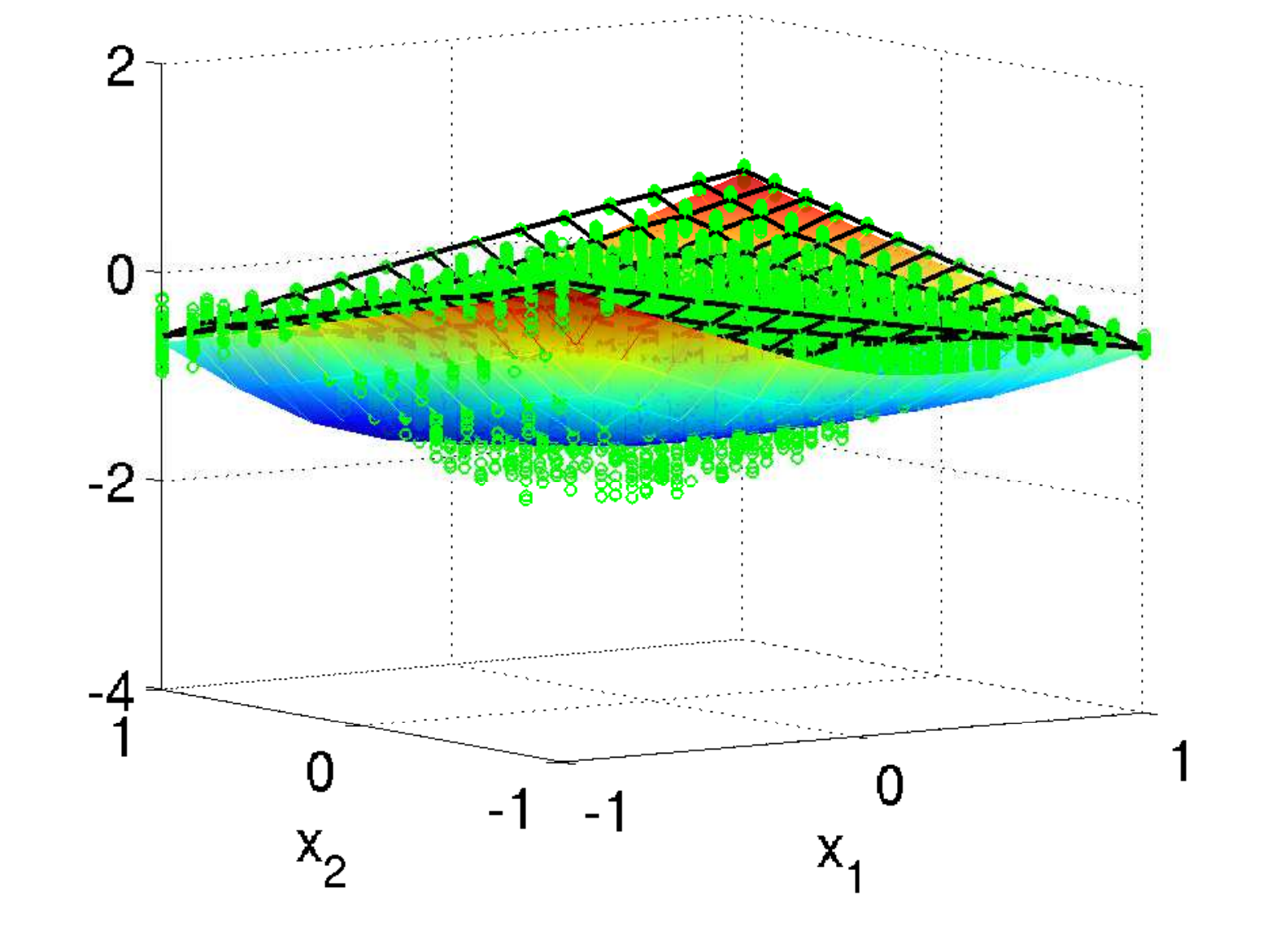}

\caption{Jensen-McCormick convex relaxations of $F$ in Example \ref{Ex: Simple JMC} (shaded surfaces) with partitions of $\ovOmega$ into 1 (top), 16 (middle), and 64 (bottom) uniform subintervals, along with simulated values of $f(\x,\bom)$ at sampled $\bom$ values ($\circ$) and a sample-average approximation of $F$ with 100 samples (black mesh).}
\label{Fig MCJ Relaxations}
\end{figure}
\end{exm}

Spatial B\&B algorithms compute a lower bound on the optimal objective value in a given subinterval $X$ by minimizing a convex relaxation of the objective function, while an upper bound is most often obtained by simply evaluating the objective at a feasible point. Theorem \ref{relaxations of RVs} provides a suitable convex relaxation for the lower bounding problem. However, in the presence of continuous RVs, computing a valid upper bound becomes nontrivial because, in general, $F$ cannot be evaluated finitely. One possible solution is to evaluate the concave relaxation $F^{cc}_{X\times\mathcal{P}}$ instead. However, this is unnecessarily conservative. Instead, rigorous upper (and lower) bounds can be computed at feasible points, without sampling error, by the following simple corollary of Theorem \ref{relaxations of RVs}.

\begin{corollary}
\label{Jensen+relaxation}
Let $\mathcal{P}=\{ \Omega_i \}_{i=1}^{n}$ be an interval partition of $\ovOmega$. For any $\x\in \ovX$, the following bounds hold:
\begin{alignat}{1}
F(\x)\geq \sumi \P{\Omega_i}\fRelax{cv}{\Omega_i}{\x}{\mathbb{E}[\bom |  \Omega_i]} ,\\
F(\x) \leq \sumi \P{\Omega_i}\fRelax{cc}{\Omega_i}{\x}{\mathbb{E}[\bom |  \Omega_i]}.
\end{alignat}
\end{corollary}
\begin{proof}
The result follows by simply applying Theorem \ref{relaxations of RVs} with the degenerate interval $X=[\x,\x]$.
\end{proof}

Clearly, the need to exhaustively partition $\ovOmega$ in Theorem \ref{relaxations of RVs} and Corollary \ref{JensTesten+relaxation} is potentially prohibitive for problems with high-dimensional uncertainty spaces. However, it is essential for obtaining deterministic bounds on $F$, rather than bounds that are only statistically valid. Moreover, Theorem \ref{relaxations of RVs} and Corollary \ref{Jensen+relaxation} provide valid bounds on $F$ for any choice of partition $\mathcal{P}$, no matter how coarse, and this has significant implications in the context of spatial B\&B. Specifically, since valid bounds are obtained with any $\mathcal{P}$, it is possible to fathom a given node $X$ with certainty using only a coarse partition of $\ovOmega$. In other words, if $X$ is proven to be infeasible or suboptimal based on such a coarse description of uncertainty, then this decision cannot be overturned at any later stage based on a more detailed representation (i.e., a finer partition). This is distinctly different from the case with sample-based bounds, where bounds based on few samples can always be invalidated by additional samples in the future. Thus, when using the bounds and relaxations of Theorem \ref{relaxations of RVs} and Corollary \ref{Jensen+relaxation} in B\&B, it is possible to refine the partition of $\overline{\Omega}$ adaptively as the search proceeds. It is therefore conceivable that, in some cases, much of the search space could be ruled out using only coarse partitions at low cost, while fine partitions are required only in the vicinity of global solutions. We leave the development and testing of such an adaptive B\&B scheme for future work. However, as a step in this direction, we now turn to the study of the convergence behavior of $F^{cv}_{X\times\mathcal{P}}$ and $F^{cc}_{X\times\mathcal{P}}$ as both $X$ and the elements of $\mathcal{P}$ are refined towards degeneracy.

\section{Convergence}
\label{Convergence}
Consider any $X\in\mathbb{I}\overline{X}$ and any interval partition $\mathcal{P}=\{\Omega_i\}_{i=1}^n$ of $\ovOmega$, and let the relaxations $F^{cv}_{X\times\mathcal{P}}$ and $F^{cc}_{X\times\mathcal{P}}$ be defined as in Theorem \ref{relaxations of RVs}. This section considers the convergence of these relaxations to $F(\x)=\mathbb{E}[f(\x,\bom)]$ as the width of $X$ and the elements $\Omega_i$ of $\mathcal{P}$ tend towards zero. We require the following assumption on the scheme of relaxations used for $f$.

\begin{assumption}
\label{Assump:f second-order pointwise convergence and C2}
The scheme of relaxations $(f_{X\times\Omega}^{cv},f_{X\times\Omega}^{cc})$ from Assumption \ref{Assump:f factorable} has second-order pointwise convergence in $\mathbb{I}\ovX\times\mathbb{I}\ovOmega$; i.e., $\exists\tau\in\mathbb{R}_+$ such that 
\begin{alignat*}{1}
\sup_{(\x,\bom)\in X\times\Omega}|f^{cc}_{X\times\Omega}(\x,\bom)-f^{cv}_{X\times\Omega}(\x,\bom)|\leq\tau w(X\times\Omega)^2,
\end{alignat*}
for all $(X\times\Omega)\in\mathbb{I}\ovX\times\mathbb{I}\ovOmega$.
\end{assumption}

Both McCormick and $\alpha$BB relaxations are known to 
satisfy Assumption \ref{Assump:f second-order pointwise convergence and C2} \cite{Bompadre:2012}. We now show that this implies a convergence bound for $F^{cv}_{X\times\mathcal{P}}$ and $F^{cc}_{X\times\mathcal{P}}$ in terms of $w(X)^2$ and the `average' square-width of partition elements $\Omega_i\in\mathcal{P}$.

\begin{lemma}
\label{Lem: Average Square Width Bound}
If Assumption \ref{Assump:f second-order pointwise convergence and C2} holds with $\tau\in\mathbb{R}_+$, then, for any $X\in\mathbb{I}\ovX$ and any interval partition $\mathcal{P}=\{\Omega_i\}_{i=1}^n$ of $\ovOmega$,
\begin{alignat*}{1}
\sup_{\x\in X}\big|F^{cc}_{X\times\mathcal{P}}(\x)- F^{cv}_{X\times\mathcal{P}}(\x)\big| \leq \tau\left[w(X)^{2}+\sumi \P{\Omega_i}w(\Omega_i)^2\right].
\end{alignat*}
\end{lemma}

\begin{proof}
Choose any $X\in\mathbb{I}\ovX$ and any $\x\in X$. Moreover, choose any interval partition $\mathcal{P}=\{\Omega_i\}_{i=1}^n$ of $\ovOmega$ and define the shorthand $\hat{\bom}_i\equiv\mathbb{E}[\bom|\Omega_i]$. From \eqref{Eq:FcvXP definition}--\eqref{Eq:FccXP definition}, we have
\begin{alignat}{1}
&\sup_{\x\in X}\big| F^{cc}_{X\times\mathcal{P}}(\x)-F^{cv}_{X\times\mathcal{P}}(\x) \big| \\
&=\sup_{\x\in X}\sumi \P{\Omega_i}\big|\fRelax{cc}{X\times \Omega_i}{\x}{\hat{\bom}_i} -\fRelax{cv}{X\times \Omega_i}{\x}{\hat{\bom}_i} \big|, \\
&\leq\sumi \P{\Omega_i}\sup_{\x\in X}\big|\fRelax{cc}{X\times \Omega_i}{\x}{\hat{\bom}_i} -\fRelax{cv}{X\times \Omega_i}{\x}{\hat{\bom}_i} \big|, \\
&\leq\sumi \P{\Omega_i}\tau w(X\times\Omega_i)^2, \\
&\leq\sumi \P{\Omega_i}\tau \left(w(X)^2+w(\Omega_i)^2\right), \\
&\leq\tau \left(w(X)^2+\sumi \P{\Omega_i}w(\Omega_i)^2\right).
\end{alignat}
\end{proof}

In order to use the relaxations $F^{cv}_{X\times\mathcal{P}}$ and $F^{cc}_{X\times\mathcal{P}}$ in a spatial B\&B algorithm, it is important that the relaxation error $\sup_{\x\in X}\big|F^{cc}_{X\times\mathcal{P}}(\x)- F^{cv}_{X\times\mathcal{P}}(\x)\big|$ converges to zero as $w(X)\rightarrow 0$ \cite{Horst:GOtext1}. However, Lemma \ref{Lem: Average Square Width Bound} suggests that this will not occur if $\mathcal{P}$ remains constant, since the term $\sum_{i=1}^n \P{\Omega_i}w(\Omega_i)^2$ will not converge to zero. However, convergence as $w(X)\rightarrow0$ can be achieved if $\mathcal{P}$ is refined appropriately as $X$ diminishes. We formalize this next.

\begin{definition}
\label{Def:Partition Map}
For every $X\in\mathbb{I}\ovX$, let $\Phi(X)=\{\Omega_i\}_{i=1}^n$ denote an interval partition of $\ovOmega$ satisfying the condition
\begin{alignat}{1}
\label{Eq:Phi Convergence Condition}
\sumi \P{\Omega_i}w(\Omega_i)^2 \leq Kw(X)^2,
\end{alignat}
for some constant $K\in\mathbb{R}_+$ that is independent of $X$. Moreover, let $(\mathcal{F}^{cv}_X,\mathcal{F}_X^{cc})_{X\in\mathbb{I}\ovX}$ be a scheme of relaxations for $F$ defined for all $X\in\mathbb{I}\ovX$ by
\begin{alignat}{1}
\label{Eq:mathcalF Def cv}
\mathcal{F}^{cv}_X (\x) &\equiv F^{cv}_{X\times\Phi(X)}(\x), \quad \forall \x\in X, \\
\label{Eq:mathcalF Def cc}
\mathcal{F}^{cc}_X (\x) &\equiv F^{cc}_{X\times\Phi(X)}(\x), \quad \forall \x\in X,
\end{alignat}
where $F^{cv}_{X\times\Phi(X)}$ and $F^{cc}_{X\times\Phi(X)}$ are defined as in \eqref{Eq:FcvXP definition}--\eqref{Eq:FccXP definition}.
\end{definition} 

\begin{theorem}
\label{Thm: Final Convergence Result}
The scheme of relaxations $(\mathcal{F}^{cv}_X,\mathcal{F}_X^{cc})_{X\in\mathbb{I}\ovX}$ for $F$ has second-order pointwise convergence; i.e., there exists $\hat{\tau}\in\mathbb{R}_+$ such that
\begin{alignat*}{1}
\sup_{\x\in X}\big|\mathcal{F}^{cc}_{X}(\x)-\mathcal{F}^{cv}_{X}(\x)\big| \leq \hat{\tau} w(X)^{2}, \quad \forall X\in\mathbb{I}\ovX.
\end{alignat*}
\end{theorem}

\begin{proof}
Let $\tau\in\mathbb{R}_+$ satisfy Assumption \ref{Assump:f second-order pointwise convergence and C2} and choose any $X\in\mathbb{I}\ovX$. By Definition \ref{Def:Partition Map} and Lemma \ref{Lem: Average Square Width Bound},
\begin{alignat*}{1}
\sup_{\x\in X}\big|\mathcal{F}^{cc}_{X}(\x)-\mathcal{F}^{cv}_{X}(\x)\big| & = \sup_{\x\in X}\big|F^{cc}_{X\times\Phi(X)}(\x)- F^{cv}_{X\times\Phi(X)}(\x)\big| \\
& \leq \tau\left[w(X)^{2}+\sumi \P{\Omega_i}w(\Omega_i)^2\right],
\end{alignat*}
where $\Phi(X)=\{\Omega_i\}_{i=1}^n$. Applying \eqref{Eq:Phi Convergence Condition},
\begin{alignat}{2}
\sup_{\x\in X}\big|\mathcal{F}^{cc}_{X}(\x)-\mathcal{F}^{cv}_{X}(\x)\big| & \leq \tau(1+K)w(X)^{2},
\end{alignat}
which proves the result with $\hat{\tau}=\tau(1+K)$.
\end{proof} 

The partitioning condition \eqref{Eq:Phi Convergence Condition} is easily satisfied in practice. For example, choosing any $K\in\mathbb{R}_+$, it is satisfied by simply partitioning $\ovOmega$ uniformly until each element satisfies $w(\Omega_i)^2\leq Kw(X)^2$. The following example demonstrates the convergence result of Theorem \ref{Thm: Final Convergence Result} using this simple scheme. Although this scheme is likely to generate much larger partitions than are necessary for convergence, we leave the issue of efficient adaptive partitioning schemes for future work.

\begin{exm}
\label{Ex:Simple Convergence}
Let $\ovX\equiv[24,26]$, let $\omega$ be uniformly distributed in $\ovOmega\equiv[10,13]$, and define $F:\ovX\rightarrow\mathbb{R}$ by $F(x)\equiv \E{f(x,\bom)}$ with
\begin{alignat}{1}
f(x,\bom)\equiv\frac{(\bom-10)^2\ln(x)+(x-5)^2}{\bom}. \nonumber
\end{alignat}
Consider the sequence of intervals $X_{\epsilon}\equiv [25-\epsilon,25+\epsilon] \subset \ovX$ with $\epsilon\rightarrow 0$, so that $w(X_{\epsilon})\rightarrow0$. For every $X_{\epsilon}$, let $\Phi(X_{\epsilon})=\{\Omega_i\}_{i=1}^{n}$ be generated by uniformly partitioning $\ovOmega$ until $w(\Omega_i)\leq 10w(X_{\epsilon})=20\epsilon$, which verifies \eqref{Eq:Phi Convergence Condition} with $K=100$. Moreover, define the relaxations $\mathcal{F}^{cv}_{X_{\epsilon}}$ and $\mathcal{F}^{cc}_{X_{\epsilon}}$ of $F$ on $X_{\epsilon}$ as in Definition \ref{Def:Partition Map}, where McCormick's relaxations are used to compute $(f^{cv}_{X\times\Omega},f^{cc}_{X\times\Omega})$ satisfying Assumptions \ref{Assump:f factorable} and \ref{Assump:f second-order pointwise convergence and C2}. Figure \ref{Fig Convergence X and Omega} shows the pointwise relaxation error $|\mathcal{F}^{cc}_{X_{\epsilon}}(x)-\mathcal{F}^{cv}_{X_{\epsilon}}(x)|$ versus $\epsilon$ for the point $x=25$. The observed slope of 2 on log-log axes indicates that the convergence is indeed second-order. \qed

\begin{figure}[h]
\centering
  \includegraphics[width=0.8\textwidth]{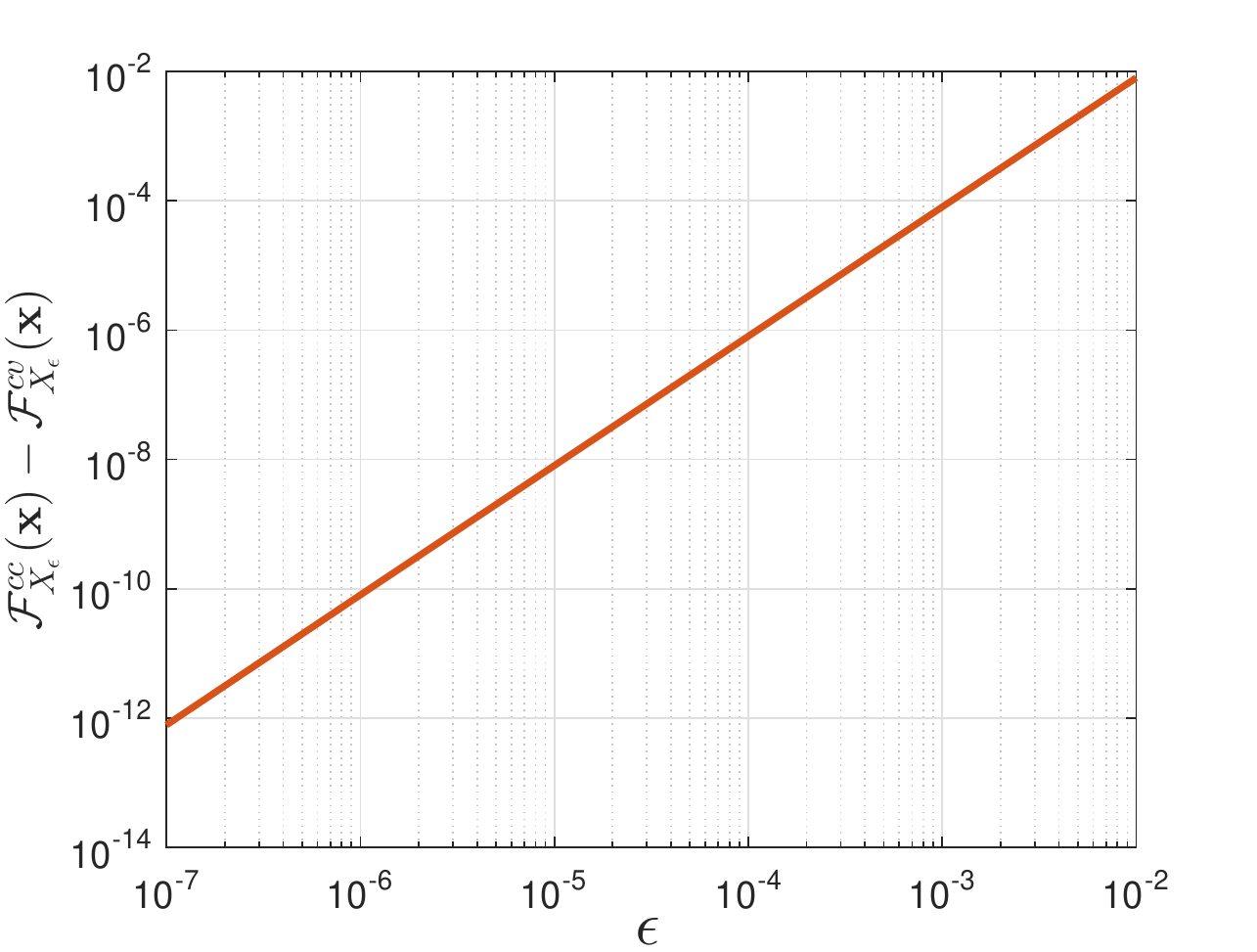}
\caption{Second-order pointwise convergence of Jensen-McCormick relaxations with respect to $\frac{1}{2}w(X_{\epsilon})=\epsilon$ for Example \ref{Ex:Simple Convergence} under a uniform $\ovOmega$ partitioning rule satisfying \eqref{Eq:Phi Convergence Condition} with $K=100$. Plotted values are for $x=\mathrm{mid}(X_{\epsilon})=25$.}
\label{Fig Convergence X and Omega}       
\end{figure}
\end{exm}

\section{Non-Uniform Random Variables}
\label{Non-Uniform RVs}
The relaxation theory presented in the previous sections puts two significant restrictions on the random variables (RVs) $\bom$. First, $\bom$ must be compactly supported in the interval $\ovOmega\in\mathbb{IR}^{n_{\omega}}$. Therefore, if we wish to use, e.g., normal random variables, we must use a truncated density function. However, this is not a major limitation since the truncated distribution can be made arbitrarily close to the original by choosing $\ovOmega$ large. Second, computing the relaxations $F^{cv}_{X\times\mathcal{P}}$ and $F^{cc}_{X\times\mathcal{P}}$ for a given partition $\mathcal{P}$ of $\ovOmega$ requires the ability to compute the probability $\mathbb{P}(\Omega)$ and conditional expectation $\mathbb{E}[\bom|\Omega]$ for any given subinterval $\Omega$ of $\ovOmega$. While this is trivial for uniform RVs, it may involve very cumbersome multidimensional integrals for other types of RVs. In this section, we address this issue in two steps. First, in \S\ref{Sec:Primitive RVs}, we compile a library of so-called \emph{primitive} RVs for which the required probabilities and expectations are easily computed by well-known formulas. Second, in \S\ref{Sec:Factorable RVs}, we extend our relaxation theory to problems with \emph{factorable} RVs, which are those RVs that can be transformed into primitive RVs by a factorable function. We then provide some general strategies for finding such transformations.

\subsection{Primitive RVs}
\label{Sec:Primitive RVs}
Recall from \S\ref{Problem Statement} that $\bom$ is a vector of random variables with PDF $p:\mathbb{R}^{n_{\omega}}\rightarrow\mathbb{R}$ compactly supported in the interval $\ovOmega\in\mathbb{IR}^{n_{\omega}}$. We will call $\bom$ a \emph{primitive} random vector if, for any $\Omega\in\mathbb{I}\ovOmega$, the quantities $\mathbb{P}(\Omega)$ and $\mathbb{E}[\bom|\Omega]$ can be computed efficiently through simple formulas. Clearly, uniformly distributed RVs are primitive. Moreover, if the elements of $\bom=(\bom_1,\ldots,\bom_{n_{\omega}})$ are independent, then several other distributions for each $\bom_i$ result in a primitive $\bom$. To see this, let $\ovOmega = \ovW_1\times\cdots\times\ovW_{n_{\omega}}$ and let $\Omega=W_1\times\cdots\times W_{n_{\omega}}\in\mathbb{I}\ovOmega$. With this notation, independence implies that 
\begin{alignat}{1}
\mathbb{P}(\Omega) &= \Pi_{i=1}^{n_{\omega}}\mathbb{P}(W_i), \\
\mathbb{E}[\bom|\Omega] &= \left(\mathbb{E}[\omega_1|W_1],\ldots,\mathbb{E}[\omega_{n_{\omega}}|W_{n_{\omega}}]\right).
\end{alignat}
Thus, $\bom$ is primitive if formulas are known for the one-dimensional probabilities $\mathbb{P}(W_i)$ and expectations $\mathbb{E}[\omega_i|W_i]$. Such formulas are well known for many common distributions. However, recall that distributions with non-interval supports must be truncated to $\ovW_i$ here. The following lemma provides a means to translate formulas for $\mathbb{P}(W_i)$ and $\mathbb{E}[\omega_i|W_i]$ for an arbitrary one-dimensional distribution to corresponding formulas for the same distribution truncated on $\ovWi$.

\begin{lemma}
\label{Lem:Truncation Lemma}
Let $\eta$ be a random variable with probability density function (PDF) $p_{\eta}:\mathbb{R}\rightarrow\mathbb{R}$, and let $\mathbb{P}_{\eta}$ and $\mathbb{E}_{\eta}$ denote the probability and expected value with respect to $p_{\eta}$, respectively. Moreover, choose any $\ovW\in\mathbb{IR}$ and let $\omega$ be a random variable with the truncated PDF $p:\mathbb{R}\rightarrow\mathbb{R}$ defined by
\begin{alignat*}{1}
p(\omega) \equiv p_{\eta}(\omega|\ovW) = \left\{ \begin{array}{lc}
p_{\eta}(\omega)/\mathbb{P}_{\eta}(\ovW) & \text{if }\omega\in \ovW \\
0 & \text{otherwise}
\end{array} \right..
\end{alignat*}
Let $\mathbb{P}$ and $\mathbb{E}$ denote the probability and expected value with respect to $p$. For any $W\in\mathbb{I}\ovW$, the following relations hold:
\begin{alignat}{1}
\label{Eq:TruncatedProb}
\mathbb{P}(W) &= \mathbb{P}_{\eta}(W)/\mathbb{P}_{\eta}(\ovW), \\
\label{Eq:TruncatedCondExp}
\mathbb{E}[\omega|W] &= \mathbb{E}_{\eta}[\omega|W],
\end{alignat}
\end{lemma}
\begin{proof}
For any $W\in\mathbb{I}\ovW$,
\begin{alignat}{1}
\mathbb{P}(W) &= \int_{W}p(\omega)d\omega =\int_{W}\frac{p_{\eta}(\omega)}{\mathbb{P}_{\eta}(\ovW)}d\omega = \frac{\mathbb{P}_{\eta}(W)}{\mathbb{P}_{\eta}(\ovW)},
\end{alignat}
which proves \eqref{Eq:TruncatedProb}. Similarly,
\begin{alignat}{1}
\mathbb{E}[\omega|W] &= \frac{1}{\mathbb{P}(W)}\int_{W}\omega p(\omega)d\omega, \\
&=\frac{1}{\mathbb{P}(W)}\int_{W} \omega\frac{p_{\eta}(\omega)}{\mathbb{P}_{\eta}(\ovW)}d\omega.
\end{alignat}
Applying \eqref{Eq:TruncatedProb}, 
\begin{alignat}{1}
\mathbb{E}[\omega|W] &= \frac{1}{\mathbb{P}_{\eta}(W)}\int_{W} \omega p_{\eta}(\omega)d\omega=\mathbb{E}_{\eta}[\bom|W].
\end{alignat}
\end{proof}

Table \ref{Table:PrimitiveDists} lists formulas for the cumulative probability distributions (CDFs) $P$ and conditional expectations $\mathbb{E}[\omega|W]$ for several common one-dimensional distributions. When necessary, distributions are truncated to an interval $\ovW=[\overline{\omega}^L,\overline{\omega}^U]$. The listed CDFs can be used to compute $\mathbb{P}(W)$ via $\mathbb{P}(W)=P(\omega^U)-P(\omega^L)$. The formulas in Table \ref{Table:PrimitiveDists} follow directly from Lemma \ref{Lem:Truncation Lemma} and standard formulas for probabilities and conditional expectations in \cite{Johnson:1994,Okasha:2014}.

\begin{table*}
\caption{Primitive one-dimensional distributions on $\ovW=[\overline{\omega}^L,\overline{\omega}^U]$ with formulas for the CDFs $P(\omega)$ and the conditional expectations $\mathbb{E}[\omega|W]$ with $\omega\in W=[\omega^L,\omega^U]\in\mathbb{I}\ovW$. Normal and Gamma distributions are truncated to $\ovW$ \cite{Okasha:2014}. Formulas for the Beta distribution are for $\ovW=[0,1]$, but arbitrary $\ovW$ can be achieved by linear transformation of the RV (see \S\ref{Sec:Factorable RVs}).}

\scalebox{0.8}{
\begin{tabular}{lccc}
\hline\noalign{\smallskip}
Distribution & Notation & $P(\omega)$ & $\mathbb{E}[\omega|W]$  \\
\noalign{\smallskip}\hline\noalign{\smallskip}

\shortstack[l]{Uniform \\  $U(\overline{\omega}^L,\overline{\omega}^U)$}  
& 
& $\dfrac{\omega-\overline{\omega}^L}{\overline{\omega}^U-\overline{\omega}^L}$
& $\dfrac{\omega^L+\omega^U}{2}$\\
\noalign{\smallskip}\hline\noalign{\smallskip}

\shortstack[l]{Truncated Normal\\ $\mathcal{N}(\mu,\sigma^2,\overline{\omega}^L,\overline{\omega}^U)$} 
& \shortstack[l]{$\phi (\xi)\equiv\frac{1}{\sqrt{2\pi}}\exp(-\frac{1}{2}\xi^2)$ \\ $\Phi(\xi)\equiv\frac{1}{2}(1+\mathrm{erf}(\xi/ \sqrt{2}))$} 
& $\dfrac{\Phi(\frac{\omega-\mu}{\sigma})-\Phi(\frac{\overline{\omega}^L-\mu}{\sigma})}{\Phi(\frac{\overline{\omega}^U-\mu}{\sigma})-\Phi(\frac{\overline{\omega}^L-\mu}{\sigma})}$ 
& $\mu + \dfrac{\phi(\frac{\omega^L-\mu}{\sigma})-\phi(\frac{\omega^U-\mu}{\sigma})}{\Phi(\frac{\omega^U-\mu}{\sigma})-\Phi(\frac{\omega^L-\mu}{\sigma})}\sigma$\\

\noalign{\smallskip}\hline\noalign{\smallskip}
\shortstack[l]{Truncated Gamma\\  $\Gamma(\alpha,\beta,\overline{\omega}^L,\overline{\omega}^U)$\\ $\alpha>0,\beta>0$}  
& $\Gamma(a,z)\equiv \int_{z}^{\infty}t^{a-1}e^{-t}dt$
& \shortstack[l]{$\dfrac{\Gamma(\alpha,\frac{\omega}{\beta})-\Gamma(\alpha,\frac{\overline{\omega}^L}{\beta})}{\Gamma(\alpha,\frac{\overline{\omega}^U}{\beta})-\Gamma(\alpha,\frac{\overline{\omega}^L}{\beta})} $}  
& \shortstack[l]{$\dfrac{\beta(\Gamma(\alpha+1,\frac{\omega^U}{\beta})-\Gamma(\alpha+1,\frac{\omega^L}{\beta}))}{\Gamma(\alpha,\frac{\omega^U}{\beta})-\Gamma(\alpha,\frac{\omega^L}{\beta})}$} \\
\noalign{\smallskip}\hline\noalign{\smallskip}

\shortstack[l]{Beta $(\overline{W}=[0,1])$ \\$B(\alpha,\beta)$\\ $\alpha>0,\beta>0$ }  
& $B(a,b,c)\equiv\int_0^c t^{a-1}(1-t)^{b-1}dt$
& $\dfrac{B(\alpha,\beta,\omega)}{B(\alpha,\beta,1)}$ 
& $\dfrac{B(1+\alpha,\beta,\omega^U)-B(1+\alpha,\beta,\omega^L)}{B(\alpha,\beta,\omega^U)-B(\alpha,\beta,\omega^L)}$\\
\noalign{\smallskip}\hline

\end{tabular}
\label{Table:PrimitiveDists}}
\end{table*}

\subsection{Factorable RVs}
\label{Sec:Factorable RVs}
We now extend our relaxation method to a flexible class of non-primitive RVs. For the following definition, recall the definition of a \emph{factorable function} discussed at the end of \S\ref{Preliminaries}.

\begin{definition}
\label{Def:FactorableRV}
The random vector $\bom$ is called \emph{factorable} if there exists an open set $\Gamma_0\subset\mathbb{R}^{n_{\omega}}$, an interval $\overline{\Gamma}\in\mathbb{I}\Gamma_0$, and a function $\psi:\Gamma_0\rightarrow \mathbb{R}^{n_{\omega}}$ such that (i) $\psi$ is continuously differentiable on $\Gamma_0$ and $\mathrm{det}(\frac{\partial\psi}{\partial\bgam}(\bgam))\neq 0$, $\forall\gamma\in\Gamma_0$; (ii) $\psi$ is one-to-one on $\Gamma_0$, and hence invertible on the image set $\psi(\Gamma_0)$; (iii) $\bom$ has zero probability density outside of $\psi(\overline{\Gamma})$; (iv) $\gamma=\psi^{-1}(\bom)$ is a primitive RV; and (v) $\psi$ can be expressed as a factorable function on $\overline{\Gamma}$.
\end{definition}

\begin{remark}
\label{Rem:FactorablRV}
A standard result in probability theory (\S6.7 in \cite{Ross:2002}) shows that, under the conditions of Definition \ref{Def:FactorableRV}, $\gamma=\psi^{-1}(\bom)$ obeys the PDF $p_{\gamma}:\mathbb{R}^{n_{\omega}}\rightarrow\mathbb{R}$ defined by 
\begin{alignat}{1}
\label{Eq:bgam PDF}
p_{\gamma}(\gamma)=\left\{\begin{array}{ll}
p(\psi(\gamma))\left|\det\left(\frac{\partial\psi}{\partial\bgam}(\gamma)\right)\right| & \text{if }\bgam\in\overline{\Gamma} \\
0 & \text{otherwise}.
\end{array}\right.,
\end{alignat}
where $p$ is the PDF of $\bom$. Thus, the fourth condition of Definition \ref{Def:FactorableRV} is equivalent to the statement that $p_{\gamma}$ is the PDF of a primitive RV.
\end{remark}

The following theorem shows that the expected value $F(\x)=\mathbb{E}[f(\x,\bom)]$ of a factorable function $f$ with respect to a factorable RV $\bom$ can be reformulated as an equivalent expected value over a primitive RV $\bgam$, and the reformulated integrand is again a factorable function.

\begin{theorem}
\label{Thm:EvalEquivalence}
Assume that $\bom$ is factorable and let $\psi$, $\overline{\Gamma}$, $\bgam$, and $p_{\gamma}$ be defined as in Definition \ref{Def:FactorableRV} and Remark \ref{Rem:FactorablRV}. Moreover, denote the expected value with respect to $p_{\bgam}$ by $\mathbb{E}_{\bgam}$. If $\hat{f}:\overline{X}\times\overline{\Gamma}\rightarrow\mathbb{R}$ is defined by $\hat{f}(\x,\bgam)=f(\x,\psi(\bgam))$, then $\hat{f}$ is a factorable function and
\begin{alignat}{1}
\mathbb{E}[f(\x,\bom)] = \mathbb{E}_{\bgam}[\hat{f}(\x,\bgam)].
\end{alignat}
\end{theorem}
\begin{proof}
By condition (v) of Definition \ref{Def:FactorableRV}, $\hat{f}$ is a composition of factorable functions, and is therefore factorable \cite{McCormick:1976}. Now, by condition (iii) of Definition \ref{Def:FactorableRV}, $p(\bom)=0$ for all $\bom\notin\psi(\overline{\Gamma})$. Therefore,
\begin{alignat*}{1}
\mathbb{E}[f(\x,\bom)] &= \int_{\psi(\overline{\Gamma})}f(\x,\bom)p(\bom)d\bom.
\end{alignat*}
Conditions (i) and (ii) imply that $\psi$ is a diffeomorphism from $\Gamma_0$ to $\psi(\Gamma_0)$. Since the boundary of $\overline{\Gamma}$ has measure zero, it follows that the boundary of $\psi(\overline{\Gamma})$ has measure zero (Lemma 18.1 in \cite{Munkres:AnalysisOnManifolds}) and can be excluded from the integral above. Moreover, $\psi$ is a valid change-of-variables from $\bom$ in the open set $\mathrm{int}(\psi(\overline{\Gamma}))$ to $\bgam$ in the open set $\psi^{-1}(\mathrm{int}(\psi(\overline{\Gamma})))$ (p.147 in \cite{Munkres:AnalysisOnManifolds}). Furthermore, the latter set is equivalent to $\mathrm{int}(\overline{\Gamma})$ by Theorem 18.2 in \cite{Munkres:AnalysisOnManifolds}. Then, applying the Change of Variables Theorem 17.2 in \cite{Munkres:AnalysisOnManifolds} and noting that the boundary of $\overline{\Gamma}$ has measure zero,
\begin{alignat*}{1}
&\mathbb{E}[f(\x,\bom)] = \int_{\overline{\Gamma}}f(\x,\psi(\bgam))p(\psi(\bgam))\left|\det\left(\frac{\partial\psi}{\partial\bgam}(\bgam)\right)\right|d\bgam.
\end{alignat*}
Then, using \eqref{Eq:bgam PDF},
\begin{alignat*}{1}
\mathbb{E}[f(\x,\bom)] = \int_{\overline{\Gamma}}f(\x,\psi(\bgam))p_{\bgam}(\bgam)d\bgam = \mathbb{E}_{\bgam}[\hat{f}(\x,\bgam)].
\end{alignat*}
\end{proof}

When $\bom$ is factorable, Theorem \ref{Thm:EvalEquivalence} implies that valid relaxations for $F(\x)=\mathbb{E}[f(\x,\bom)]$ can be obtained by relaxing the equivalent representation $F(\x)=\mathbb{E}_{\gamma}[\hat{f}(\x,\bgam)]$. Relaxations of the latter expression can be readily computed as described in \S\ref{Relaxations for Expected-Value Functions} since $\hat{f}$ is factorable and $\bgam$ is a primitive RV. Specifically, letting $\mathcal{P}=\{\Gamma_i\}_{i=1}^{n}$ denote an interval partition of $\overline{\Gamma}$, valid relaxations are given by
\begin{alignat}{1}
\label{Eq:GammaRel cv}
F^{cv}_{X\times\mathcal{P}}(\x)\equiv\sum_{i=1}^{n}\mathbb{P}_{\gamma}(\Gamma_i)\hat{f}^{cv}_{X\times\Gamma_i}(\x,\mathbb{E}_{\bgam}[\bgam|\Gamma_i]), \\
\label{Eq:GammaRel cc}
F^{cc}_{X\times\mathcal{P}}(\x)\equiv\sum_{i=1}^{n}\mathbb{P}_{\gamma}(\Gamma_i)\hat{f}^{cc}_{X\times\Gamma_i}(\x,\mathbb{E}_{\bgam}[\bgam|\Gamma_i]).
\end{alignat}
Thus, relaxations can be computed without explicitly computing $\mathbb{P}(\Omega)$ and $\mathbb{E}[\bom|\Omega]$ for non-primitive distributions. 

In the following subsections, we discuss several methods for computing the factorable transformation $\psi$ required by Definition \ref{Def:FactorableRV}. We first consider one-dimensional transformations that can be used to handle RVs $\bom$ with more general distributions than in Table \ref{Table:PrimitiveDists}, but still restricted to having independent elements. Subsequently, we show how RVs with independent elements can be transformed into RVs with any desired covariance matrix.

\subsection{The Inverse Transform Method}
Let $P:\mathbb{R}\rightarrow[0,1]$ be the cumulative distribution function (CDF) for $\omega$. It is well known that the RV defined by $\gamma = P(\omega)$ is uniformly distributed on $[0,1]$ \cite{Ripley2006}. Thus, $\gamma$ is a primitive RV, and $\bom=P^{-1}(\gamma)$ is a factorable RV as per Definition \ref{Def:FactorableRV} under the following assumptions.

\begin{assumption}
\label{Assum:CDF}
Let $\omega$ be compactly supported in the interval $\ovW\in\mathbb{IR}$. Assume that $P$ is continuously differentiable on $\ovW$ and satisfies $\frac{dP}{d\omega}(\omega)=p(\omega)>0$ for all $\omega\in\ovW$. Moreover, assume that $P^{-1}$ is a factorable function.
\end{assumption}

Under Assumption \ref{Assum:CDF}, it is always possible to define an open set $W_0$ containing $\ovW$ and a continuously differentiable function $\psi^{-1}:W_0\rightarrow\mathbb{R}$ that agrees with $P$ on $\ovW$ and satisfies $\frac{d\psi^{-1}}{d\omega}(\omega)>0$ for all $\omega\in W_0$. This last condition ensures that $\psi^{-1}$ has an inverse $\psi$ defined on $\Gamma_0\equiv\psi^{-1}(W_0)$, which contains the interval $\overline{\Gamma}\equiv\psi^{-1}(\ovW)=P(\ovW)=[0,1]$. By Theorem 8.2 in \cite{Munkres:AnalysisOnManifolds}, $\Gamma_0$ is open and $\psi$ is continuously differentiable on $\Gamma_0$. Moreover, an application of the chain rule to the identity $\psi^{-1}(\psi(\gamma))=\gamma$ gives
\begin{alignat}{1}
\label{Eq:PsiChainRule}
\frac{d\psi}{d\gamma}(\gamma)=\left[\frac{d\psi^{-1}}{d\omega}(\psi(\gamma))\right]^{-1}>0, \quad \gamma\in \Gamma_0.
\end{alignat}
Thus, conditions (i) and (ii) of Definition \ref{Def:FactorableRV} are satisfied. Condition (iii) states that $\psi(\overline{\Gamma})=P^{-1}([0,1])$ contains the support of $\omega$, which holds by definition of $P$. Condition (iv) holds because, by Remark \ref{Rem:FactorablRV} and \eqref{Eq:PsiChainRule},
\begin{alignat}{1}
p_{\gamma}(\gamma)&=p(\psi(\gamma))\left[\frac{d\psi^{-1}}{d\omega}(\psi(\gamma))\right]^{-1} \\
&=p(\psi(\gamma))\left[\frac{dP}{d\omega}(\psi(\gamma))\right]^{-1}, \\
&=p(\psi(\gamma))\left[p(\psi(\gamma))\right]^{-1}, \\
&=1,
\end{alignat}
for all $\gamma\in\overline{\Gamma}$, and $p_{\gamma}(\gamma)=0$ otherwise. Thus, $\gamma$ is uniform. Finally, condition (v) of Definition \ref{Def:FactorableRV} holds by the assumption that $P^{-1}$ is factorable. Thus, $\bom$ is a factorable RV under the transformation $\bom=P^{-1}(\gamma)$.

For example, let $\omega \sim W(\alpha,\beta,\overline{\omega}^L,\overline{\omega}^U)$ be a truncated Weibull RV with coefficients $\alpha$ and $\beta$. The standard (untruncated) Weibull CDF and inverse CDF are given by
\begin{alignat}{1}
P_{\eta}(\eta) &= 1-e^{-(\eta/ \alpha)^{\beta}}, \\
P_{\eta}^{-1}(\gamma) &=\alpha \sqrt[\beta]{-\ln(1-\gamma)}.
\end{alignat}
The CDF of the truncated distribution is given by
\begin{alignat}{1}
P(\omega) = \frac{P_{\eta}(\omega)-P_{\eta}(\overline{\omega}^L)}{P_{\eta}(\overline{\omega}^U)-P_{\eta}(\overline{\omega}^L)}, \quad \forall\omega\in\ovW,
\end{alignat}
and hence
\begin{alignat}{1}
P^{-1}(\gamma) &= P_{\eta}^{-1}\left[P_{\eta}(\overline{\omega}^L)+(P_{\eta}(\overline{\omega}^U)-P_{\eta}(\overline{\omega}^L))\gamma\right], \\
& =\alpha \sqrt[\beta]{-\ln\left[e^{-(\frac{\overline{\omega}^L}{\alpha})^{\beta}}+(e^{-(\frac{\overline{\omega}^U}{\alpha})^{\beta}}-e^{-(\frac{\overline{\omega}^L}{\alpha})^{\beta}})\gamma\right]}. \nonumber
\end{alignat}
This gives the desired factorable transformation $\omega = \psi(\gamma) = P^{-1}(\gamma)$. Similar transformations for several common distributions are collected in  Table \ref{Table:FactorableRvs}.

\begin{table*}
\caption{Some common RVs $\omega$ that are factorable in the sense of Definition \ref{Def:FactorableRV} with $\gamma$ uniformly distributed on $\overline{\Gamma}=[0,1]$ and factorable transformations $\psi(\gamma)=P^{-1}(\gamma)$ given by inverse CDFs. All RVs $\omega$ are truncated on the interval $\ovW=[\overline{\omega}^L,\overline{\omega}^U]$. $P_{\eta}^{-1}(\gamma)$ is the inverse CDF of the corresponding untruncated RV, so that $P^{-1}(\gamma) = P_{\eta}^{-1}\left[P_{\eta}(\overline{\omega}^L)+(P_{\eta}(\overline{\omega}^U)-P_{\eta}(\overline{\omega}^L))\gamma\right]$.}

\scalebox{0.8}{
\begin{tabular}{lccc}
\hline\noalign{\smallskip}\noalign{\smallskip}
Distribution &  & $P_{\eta}^{-1}(\gamma)$ & $P^{-1}(\gamma)$ \\
\noalign{\smallskip}\hline\noalign{\smallskip}

\shortstack[l]{Truncated\\Exponential}
& \shortstack[l]{$Exp(\lambda,\overline{\omega}^L,\overline{\omega}^U)$ \\ $\lambda > 0$}  
& $-\frac{1}{\lambda}\ln (1-\gamma)$
& $-\frac{1}{\lambda}\ln [e^{-\lambda\overline{\omega}^L}+(e^{-\lambda\overline{\omega}^U}-e^{-\lambda\overline{\omega}^L})\gamma]$\\
\noalign{\smallskip}\hline\noalign{\smallskip}

\shortstack[l]{Truncated\\Weibull}
& \shortstack[l]{$W(\alpha,\beta,\overline{\omega}^L,\overline{\omega}^U)$ \\ $\alpha,\beta > 0$}  
& $\alpha \sqrt[\beta]{-\ln(1-\gamma)}$
&$\alpha \sqrt[\beta]{-\ln\left[e^{-(\frac{\overline{\omega}^L}{\alpha})^{\beta}}+(e^{-(\frac{\overline{\omega}^U}{\alpha})^{\beta}}-e^{-(\frac{\overline{\omega}^L}{\alpha})^{\beta}})\gamma\right]}$\\
\noalign{\smallskip}\hline\noalign{\smallskip}

\shortstack[l]{Truncated\\Cauchy}
& \shortstack[l]{$C(\alpha,\beta,\overline{\omega}^L,\overline{\omega}^U)$}  
& $\beta \tan((\gamma-\frac{1}{2})\pi)+\alpha$
& \shortstack[l]{$\beta\tan[\arctan(\frac{\overline{\omega}^L-\alpha}{\beta})+(\arctan(\frac{\overline{\omega}^U-\alpha}{\beta})-\arctan(\frac{\overline{\omega}^L-\alpha}{\beta}))\gamma]+\alpha$} \\
\noalign{\smallskip}\hline\noalign{\smallskip}

\shortstack[l]{Truncated\\Rayleigh}
& \shortstack[l]{$Ra(\sigma,\overline{\omega}^L,\overline{\omega}^U)$ \\ $\sigma>0$}  
& $\sqrt{-2\sigma^2 \ln (1-\gamma)}$
& $\sqrt{-2\sigma^2 \ln [e^{-\frac{(\overline{\omega}^L)^2}{2\sigma^2}}+(e^{-\frac{(\overline{\omega}^U)^2}{2\sigma^2}}-e^{-\frac{(\overline{\omega}^L)^2}{2\sigma^2}})\gamma]}$\\
\noalign{\smallskip}\hline\noalign{\smallskip}

\shortstack[l]{Truncated\\Pareto}
& \shortstack[l]{$Pa(m,\alpha,\overline{\omega}^L,\overline{\omega}^U)$ \\ $m,\alpha>0$}  
& $m(1-\gamma)^{-\frac{1}{\alpha}}$
& $m[(\frac{m}{\overline{\omega}^L})^{\alpha}+((\frac{m}{\overline{\omega}^U})^{\alpha}-(\frac{m}{\overline{\omega}^L})^{\alpha})\gamma]^{-\frac{1}{\alpha}}$\\
\noalign{\smallskip}\hline

\end{tabular}
\label{Table:FactorableRvs}}
\end{table*}

\subsection{Other Transformations}
For many common RVs, explicit transformations from simpler RVs have been developed for the purpose of generating samples computationally. One example is the Box-Muller transformation, which transforms the two-dimensional uniform RV $\bgam$ on $(0,1]\times[0,1]$ into two independent standard normal RVs $\bom$ via
\begin{alignat}{1}
\label{Eq:BoxMuller}
\omega_1 & = \sqrt{-2\ln \gamma_1}\cos (2\pi \gamma_2), \\
\omega_2 & = \sqrt{-2\ln \gamma_1}\sin (2\pi \gamma_2). \nonumber
\end{alignat}
According to \cite{Martino2012}, the Box-Muller transform can also be used to generate independent bivariate normal RVs truncated on a disc of radius $\overline{r}$; i.e., with $\ovOmega \equiv \{\bom\in\mathbb{R}^{2}:\omega_1^2+\omega_2^2\leq\overline{r}^2\}$. This is accomplished by simply setting $\overline{\gamma}_1^L=\exp(-\frac{\overline{r}^2}{2})$, $\overline{\Gamma}=[\overline{\gamma}_1^L,1]\times[0,1]$, and defining $\gamma$ on $\overline{\Gamma}$ by \eqref{Eq:BoxMuller}, which is clearly factorable.

Many other factorable transformations can be readily devised. Naturally, we can replace log-normal RVs by the log of normal RVs, $\chi^2$ RVs by summations of squared standard normal RVs, etc.

\subsection{Dependent RVs}
\label{Sec:Dependent RVs}
Using the techniques outlined in the previous subsections, factorable random vectors can be generated with elements obeying a wide variety of common distributions, but all elements must be independent. Let $\hat{\bom}$ denote such a random vector, so that $\mathrm{Cov}(\hat{\bom})=\bs{I}$. In this subsection, we show that another factorable (in fact linear) transformation can be used to generate a random vector $\bom$ from $\hat{\bom}$ with any desired mean $\bs{d}$ and positive definite (PD) covariance matrix $\bs{C}$. For any such $\bs{C}$, positive definiteness implies that there exists a unique PD matrix $\bs{C}^{1/2}$ such that $(\bs{C}^{1/2})^2=\bs{C}$. Thus, defining $\bom = \bs{d}+\bs{C}^{1/2}(\hat{\bom}-\mathbb{E}[\hat{\bom}])$, we readily obtain $\mathbb{E}[\bom]=\bs{d}$ and 
\begin{alignat}{1}
\mathrm{Cov}(\bom) &= \mathbb{E}[(\bom-\mathbb{E}[\bom])(\bom-\mathbb{E}[\bom])^{\mathrm{T}}], \\
&=  \mathbb{E}[\bs{C}^{1/2}(\hat{\bom}-\mathbb{E}[\hat{\bom}])(\hat{\bom}-\mathbb{E}[\hat{\bom}])^{\mathrm{T}}(\bs{C}^{1/2})^{\mathrm{T}}], \\
&=  \bs{C}^{1/2}\mathrm{Cov}(\hat{\bom})\bs{C}^{1/2}, \\
&= \bs{C}, 
\end{alignat}
as desired. Note that this linear transformation can be applied directly to any of the primitive RVs listed in Table \ref{Table:PrimitiveDists} (e.g., to obtain a multivariate normal distribution with mean $\bs{d}$ and covariance $\bs{C}$) or to more general factorable RVs constructed via the methods in the previous subsections.

\begin{exm}
\label{Ex:Reactors}
To demonstrate the use of non-primitive factorable RVs, we consider relaxing the objective function of the following stochastic optimization problem modified from \cite{Ryoo:1995}, which considers the optimal design of two consecutive continuous stirred-tank reactors of volumes $x_1$ and $x_2$ carrying out the reactions $A\rightleftharpoons B\rightleftharpoons C$:
\begin{align*}
\min_{\x} \ \ &-\mathbb{E}\left[\frac{k_{f,2} x_2(1+k_{r,1} x_1)+k_{f,1} x_1(1+k_{f,2} x_2)}{(1+k_{f,1} x_1)(1+k_{f,2} x_2)(1+k_{r,1} x_1)(1+k_{r,2} x_2)}\right]\\
\text{s.t.} \ \  &x_1^{0.5} + x_2^{0.5} \leq 4\\
&0 \leq x_1,x_2\leq 16\\
&k_{r,1} = 0.99k_{f,1}\\
&k_{r,2} = 0.90k_{f,2} \\
&\begin{bmatrix}
    k_{f,1}\\
    k_{f,2}
\end{bmatrix}
=
\begin{bmatrix}
	0.097 \\
	0.039
\end{bmatrix}
+
\begin{bmatrix}
    0.0072 \ 0.0004 \\
    0.0008 \ 0.0036
\end{bmatrix}
\begin{bmatrix}
    \gamma_1\\
    \gamma_2
\end{bmatrix} \\
&\gamma_1 \sim \mathcal{N}(\mu_1,\sigma_1^2,\overline{\gamma}_1^{L},\overline{\gamma}_1^{U}) \\
&\gamma_2 \sim \mathcal{N}(\mu_2,\sigma_2^2,\overline{\gamma}_2^{L},\overline{\gamma}_2^{U})
\end{align*}
The objective is to maximize the expected value of the concentration of product B in the exit stream of the second reactor. We assume that the forward reaction rates $(k_{f,1},k_{f,2})$ follow a bivariate normal distribution with
\begin{alignat}{1}
\bs{d}=\mathbb{E}\left(\begin{bmatrix} k_{f,1} \\ k_{f,2} \end{bmatrix}\right) &=\begin{bmatrix} 0.097 \\ 0.039 \end{bmatrix}, \\
\bs{C}^{1/2}=\sqrt{\mathrm{Cov}\left(\begin{bmatrix} k_{f,1} \\ k_{f,2} \end{bmatrix}\right)} &=\begin{bmatrix} 0.0072 \ 0.0004 \\ 0.0008 \ 0.0036 \end{bmatrix}.
\end{alignat}
To avoid computing probabilities and conditional expected values of this distribution over subintervals, we introduce a primitive random vector $\gamma$ consisting of two independent random variables $(\gamma_1,\gamma_2)$ satisfying standard normal distributions ($\mu_1=\mu_2=0$ and $\sigma_1=\sigma_2=1$) truncated at plus and minus five standard deviations, giving the truncation bounds
\begin{alignat}{1}
[\overline{\gamma}_1^{L},\overline{\gamma}_1^{U}] = [\overline{\gamma}_2^{L},\overline{\gamma}_2^{U}] = [-5,5].
\end{alignat}
After truncation, $\mathrm{Var}(\gamma_1)=\mathrm{Var}(\gamma_2)=0.9999\approx 1$. We then transform $(\gamma_1,\gamma_2)$ into a bivariate normal RV $(k_{f,1},k_{f,2})$ with the desired mean and covariance as discussed in \S\ref{Sec:Dependent RVs}. With these definitions, Figure \ref{Fig MCJ Relaxations for NonUniform RVs} shows convex relaxations for the objective function on $X\equiv[2.5,4]\times[2.5,4]$ computed as in \eqref{Eq:GammaRel cv} with partitions $\mathcal{P}=\{\Gamma_i\}_{i=1}^{n_{\omega}}$ of $\overline{\Gamma}$ consisting of 1, 16, and 64 elements. The required relaxations of $\hat{f}(\x,\bgam)=f(\x,\bs{d}+\bs{C}^{1/2}\bgam)$ on each $X\times\Gamma_i$ were computed using McCormick relaxations \cite{McCormick:1976}. Figure \ref{Fig MCJ Relaxations for NonUniform RVs} shows that the relaxation computed with a partition consisting of only one element is fairly weak, but improves significantly with 16 elements. On the other hand, the additional improvement achieved with 64 elements is minor.

\begin{figure}
\centering
\includegraphics[width=0.5\textwidth,height=0.4\textwidth]{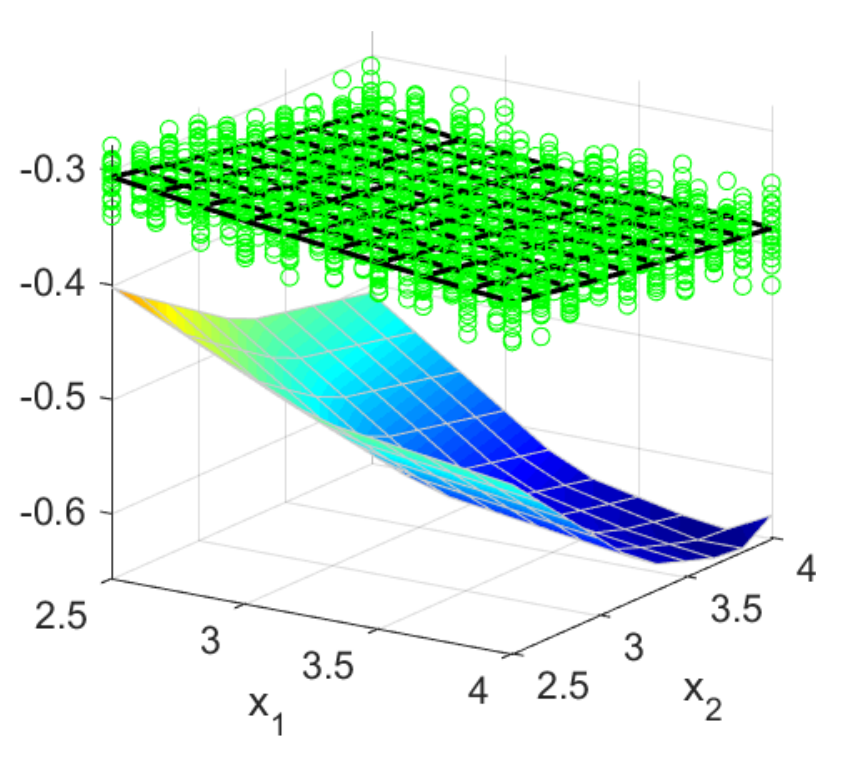}
\includegraphics[width=0.5\textwidth,height=0.4\textwidth]{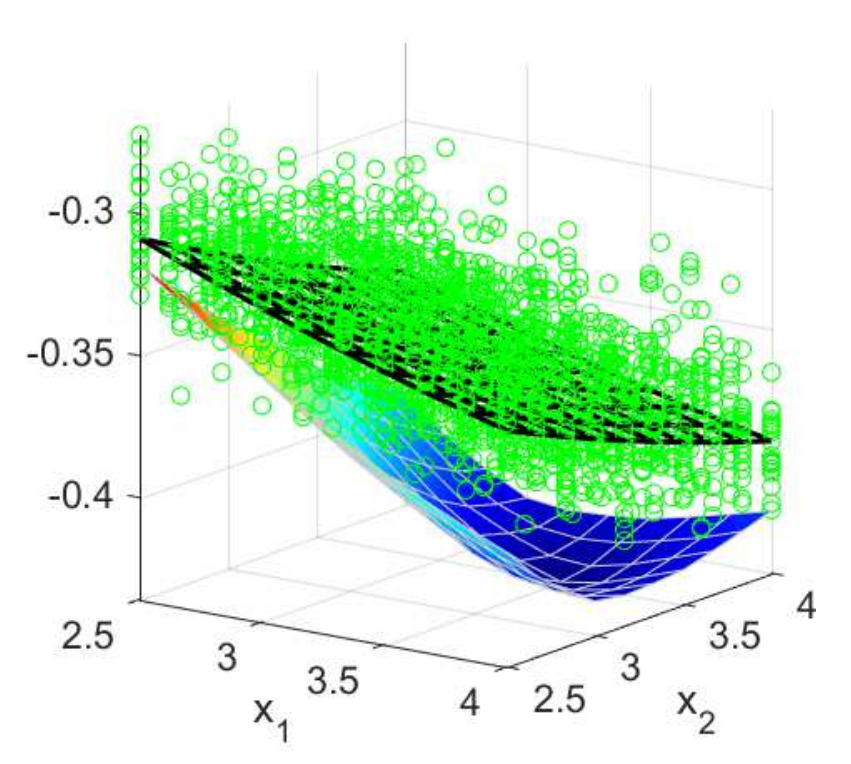}
\includegraphics[width=0.5\textwidth,height=0.4\textwidth]{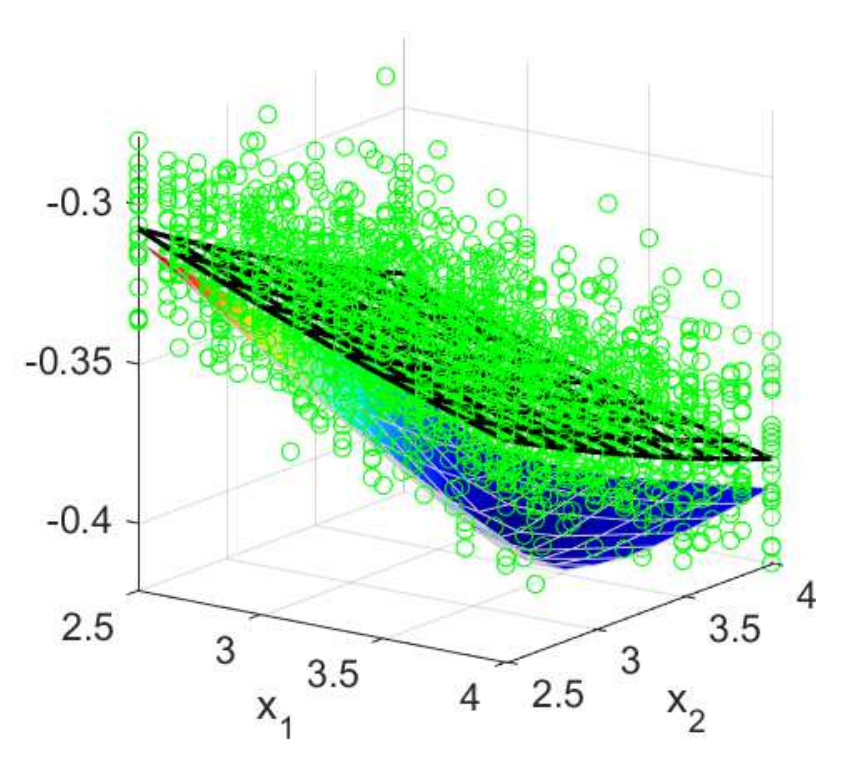}

\caption{Jensen-McCormick convex relaxations of the objective function $F$ in Example \ref{Ex:Reactors} (shaded surfaces) with partitions of $\overline{\Gamma}$ into 1 (top), 16 (middle), and 64 (bottom) uniform subintervals, along with simulated values of $f(\x,\bom)$ at sampled $\bom$ values ($\circ$) and a sample average approximation of $F$ with 100 samples (black mesh).}
\label{Fig MCJ Relaxations for NonUniform RVs}
\end{figure}

Figure \ref{Fig Convergence X and Omega for Reactor Example} shows the pointwise convergence of the computed relaxations on a sequence of intervals $X_{\epsilon} \equiv [5-\epsilon,5+\epsilon]\times [6-\epsilon,6+\epsilon]$ as $\epsilon\rightarrow0$. Specifically, Figure \ref{Fig Convergence X and Omega for Reactor Example} shows the convergence of the scheme of relaxations $(\mathcal{F}_{X}^{cv},\mathcal{F}_{X}^{cc})$ defined analogously to \eqref{Eq:mathcalF Def cv}--\eqref{Eq:mathcalF Def cc} by imposing an $X$-dependent uniform partitioning rule to $\overline{\Gamma}$ such that
\begin{alignat}{1}
\label{Eq:Gamma Partition Convergence Condition}
\sumi \P{\Gamma_i}w(\Gamma_i)^2 \leq Kw(X)^2.
\end{alignat}
Again, the observed slope of 2 on log-log axes verifies the theoretical second-order pointwise convergence ensured by Theorem \ref{Thm: Final Convergence Result}. \qed

\begin{figure}[h] 
\centering
  \includegraphics[width=0.8\textwidth]{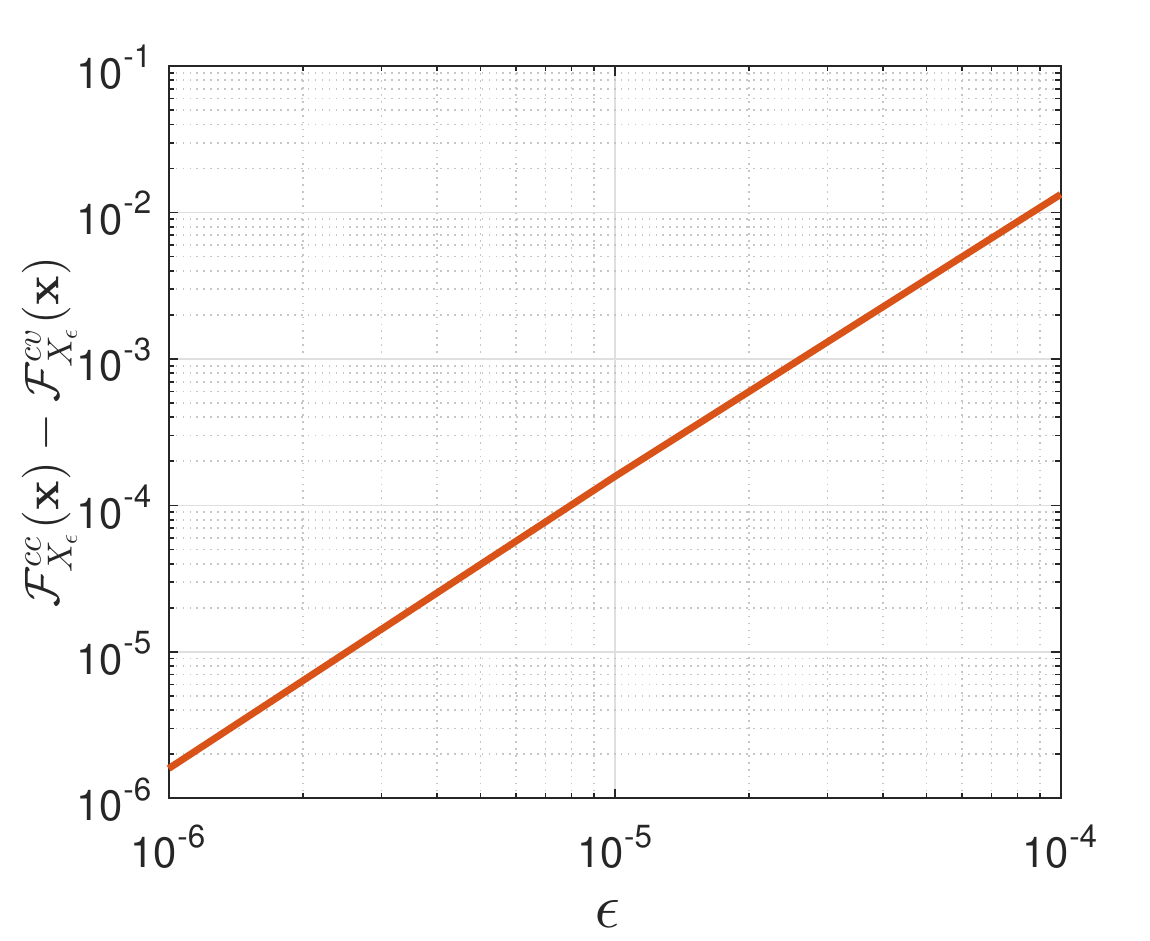}
\caption{Second-order pointwise convergence of Jensen-McCormick relaxations with respect to $\frac{1}{2}w(X_{\epsilon})=\epsilon$ for Example \ref{Ex:Reactors} under a uniform $\overline{\Gamma}$ partitioning rule satisfying \eqref{Eq:Gamma Partition Convergence Condition} with $K=10^8$. Plotted values are for $\x=\mathrm{mid}(X_{\epsilon})=(5,6)$.}
\label{Fig Convergence X and Omega for Reactor Example}       
\end{figure}
\end{exm}

\section{Conclusions}
\label{Conclusions}
In this article, we developed a new method for computing convex and concave relaxations of nonconvex expected-value functions over continuous random variables (RVs). These relaxations can provide rigorous lower and upper bounds for use in spatial branch-and-bound (B\&B) algorithms, thereby enabling the global solution of nonconvex optimization problems subject to continuous uncertainties (e.g., process yields, renewable renewable resources, product demands, etc.). Importantly, these relaxations are not sample-based. Instead, they make use of an exhaustive partition of the uncertainty set. As a consequence, they can be evaluated finitely, even when the original expected-value cannot be. Empirical results with simple uniform partitions showed that tight relaxations can be obtained with fairly coarse partitions. Moreover, when the uncertainty partition is refined appropriately, we established second-order pointwise convergence of the relaxations to the true expected value as the relaxation domain tends to a singleton. Such convergence is critical for ensuring finite termination of spatial B\&B and avoiding the cluster effect. Finally, using the notion \emph{factorable RVs}, we extended our relaxation technique to a wide variety of multivariate probability distributions in a manner that avoids the need to compute any difficult multidimensional integrals. In a forthcoming paper, we plan to develop a complete spatial B\&B algorithm for nonconvex optimization problems with continuous uncertainties by combining the relaxations developed here with efficient, adaptive uncertainty partitioning strategies.

\footnotesize
\bibliographystyle{spmpsci_unsrt_nodoi}
\bibliography{JMC}

\end{document}